\documentclass[12pt,twoside]{amsart}
\usepackage[margin=3cm]{geometry}
\usepackage[colorlinks=false]{hyperref}
\usepackage[english]{babel}
\usepackage{graphicx}
\usepackage{float}
\usepackage{amsmath,amsfonts,amssymb,amsthm}
\usepackage{lipsum}
\usepackage[T1]{fontenc}
\usepackage{fourier}
\usepackage{color}
\usepackage{esint}
\usepackage{caption}
\usepackage{siunitx}
\usepackage{ccicons}
\usepackage{subcaption}
\makeatletter
\def\blfootnote{\xdef\@thefnmark{}\@footnotetext}
\makeatother
\newcommand{\R}{\mathbb{R}}

\newcommand{\Ric}{\operatorname{Ric}}

\newcommand{\Vol}{\operatorname{Vol}}

\newcommand{\hess}{\operatorname{Hess}}

\newcommand{\grad}{\nabla}

\newcommand{\eps}{\varepsilon}
\newcommand{\dd}{\,\mathrm{d}}

\newcommand\ccnote{
    \blfootnote{The research is supported by  NSFC 11721101 and NSFC 12531001.}
}
\allowdisplaybreaks
\usepackage[export]{adjustbox}
\numberwithin{equation}{section}
\usepackage{setspace}
\renewcommand{\le}{\leqslant}
\renewcommand{\leq}{\leqslant}
\renewcommand{\ge}{\geqslant}
\renewcommand{\geq}{\geqslant}
\renewcommand{\mathbb}{\varmathbb}
\usepackage{fancyhdr}
\newtheorem{theorem}{Theorem}[section]
\newtheorem{lemma}[theorem]{Lemma}
\newtheorem{corollary}[theorem]{Corollary}
\newtheorem{proposition}[theorem]{Proposition}

\usepackage{comment}
\usepackage{enumerate}
\usepackage{hyperref}
\usepackage[alphabetic, msc-links, backrefs]{amsrefs}

\usepackage{amsthm}

\allowdisplaybreaks

\input epsf
\def\begfig {
\begin{figure}
\small }
\def\endfig {
\normalsize
\end{figure}
}

\begin{document}
\thispagestyle{empty}

\ccnote

\title[]{Quantitative Rigidity of Ricci Nonnegative Manifolds via the Regular Green Constant}

\author{Zhelei Huang}
\author{Bijia Luo}
\author{Jie Zhou}
\address{Zhelei Huang, School of Mathematical Sciences, Capital Normal University, 105 West Third Ring Road North, Haidian District, Beijing, 100048, China}
\email{2260501005@cnu.edu.cn}

\address{Bijia Luo, School of Mathematical Sciences, Capital Normal University,  105 West Third Ring Road North, Haidian District, Beijing, 100048, China}
\email{2250502153@cnu.edu.cn}

\address{Jie Zhou, School of Mathematical Sciences, Capital Normal University, 105 West Third Ring Road North, Haidian District, Beijing, 100048, China}
\email{zhoujiemath@cnu.edu.cn}
\maketitle

\noindent \textbf{Abstract.} \textit{For a Ricci non-negative non-parabolic manifold $(M^{n\ge3},g)$, assuming  the minimal normalized Green function  admits the asymptotic expansion 
$G=r^{2-n}+H_p+O_1(r)$, we prove the scaling invariant defect $H_pR^{n-2}$  controls the  geometry of the geodesic ball quantitatively, $\frac{1}{R}d_{GH}(B_p(R),B_0^{\mathbb{R}^n}(R))\le C_n \left(H_p R^{n-2}\right)^{\frac{1}{88(n+1)(4n+1)}}$.
}
\vskip0.3cm

\vspace{0.5cm}



\tableofcontents

\section{Introduction}
\label{sec:introduction}

Let $(M^n,g)$ be a smooth, complete, non-compact, non-parabolic
Riemannian manifold of dimension $n\ge 3$ satisfying
\[
    \operatorname{Ric}\ge 0.
\]
By the work of Varopoulos\cite{Varopoulos81} and Li-Yau\cite{LiYau1986},  $(M^n,g)$ admits a   unique normalized minimal positive Green function. More precisely, fix a point $p\in M$, set $r(x)=d(p,x)$. Then, there exists a unique positive Green function $G(p,\cdot)$ such that 
\[
    -\Delta G(p,\cdot)=(n-2)\sigma_{n-1}\delta_p,
    \qquad
    \lim_{x \to p}G(p,x)r(x)^{n-2}=1 
    \quad\text{ and }\quad \lim
    _{x\to \infty} G(p,x)=0.
\]
In \cite{Colding2012}, Colding  considered the function 
$$b=G^{\frac{1}{2-n}}$$
and studied the monotonicity of geometric quantities along the level set of $b$. 
 Especially, Colding introduced the area funcitonal 
 $$A(s):=s^{1-n}\int_{\{b=s\}}|\nabla b|^3\dd A$$
 and proved his third monotonicity formula 
 \begin{equation}\label{eq:colding-third-monotonicity}
 s_2^{3-n}A'(s_2) -s_1^{3-n}A'(s_1)=\frac12\int_{\{s_1\leq b\leq s_2\}}D\,b^{2-2n}\dd V,
\end{equation}
where
\[D:=\left|\operatorname{Hess}b^2-\frac{\Delta b^2}{n}g\right|^2+\operatorname{Ric}\bigl(\nabla b^2,\nabla b^2\bigr).
\]
 As  one of the applications of this monotonicity formula, Colding proved\cite[Corollary 2.22]{Colding2012} that any three dimensional non-parabolic manifold with nonnegative Ricci curvature must be $\mathbb{R}^3$ if $|\nabla b|^2$ is $C^2$. To understand this rigidity phenomenon coming from regularity, Huang and Xiong\cite{HuangXiongPole2026} studied the  asymptotic  expansion of the  Green function.  They proved\cite[Theorem 1.2]{HuangXiongPole2026} that under some vanishing conditions(in dimension three, the conditions are empty) of the curvature at $p$, the minimal Green function admits the expansion 
\begin{equation}\label{eq:intro_green_expansion}
    G(p,x)=r^{2-n}+H_p+O_1(r).
\end{equation}
Here $H_p\ge 0$ is a constant  and  $O_1(r)$ is a function which  vanishes at $p$ and has bounded gradient.  Moreover, in odd dimensions, the $C^{n-2}$ regularity of $|\nabla b|^2$ means $H_p=0$, which further  implies  $M^n=\mathbb{R}^n$\cite[Theorem 1.3]{HuangXiongPole2026}.

In the present paper, we continue to consider the quantification of the rigidity theorem. The topic is closely related to quantitative rigidity results in Cheeger-Colding theory. By the volume comparison theorem and its rigidity, the volume difference of a manifold and the comparing model is usually chosen as the defect for a quantitative rigidity theorem. For a manifold with $\mathrm{Ric}\ge n-1$ and almost maximal volume,  Colding \cite{Colding96} proved the manifold is Gromov-Hausdorff close to $\mathbb{S}^n$. For geodesic ball in a  manifold with $\mathrm{Ric}\ge 0$ and almost maximal volume, Colding\cite[Theorem 0.8]{Colding1997} pointed out the geodesic ball is close to the Euclidean ball by similar proof as in \cite{Colding96}. Cheeger\cite[Theorem 9.69]{Cheeger01} gave a detailed proof of\cite[Theorem 0.8]{Colding1997}  by compactness argument using the " almost volume cone implies almost metric cone" theorem  and almost splitting theorem in \cite{CheegerColding1996}.  In our setting, the rigidity comes from the regularity of $|\nabla b|^2$ or the vanishing of the constant $H_p$.   Since the regularity of $|\nabla b|^2$ corresponds to the vanishing of  $H_p$ only in odd dimensions, and the vanishing of $H_p$ implies the rigidity in all dimensions,  we prefer to choose $H_p$ as the quantitative defect and ask how close $(M^n,g)$ is  to $\mathbb{R}^n$ when $H_p$ is sufficiently small. Since $M$ is non-compact, we use the pointed Gromov-Hausdorff distance $\frac{1}{R}d_{\mathrm{GH}}(B_p(R),B_0^{\mathbb{R}^n}(R))$ to measure the distance between  $M$ and $\mathbb{R}^n$.    The following is the main answer. 
\begin{theorem}[Quantitative rigidity]
\label{main}
Let $(M^{n\ge 3},g)$ be a complete non-parabolic manifold with non-negative Ricci curvature and $G(p,\cdot)$ be the normalized minimal Green function at the base point $p$. 
Assume
\[
    G(p,x)=r(x)^{2-n}+H_p+O_1(r(x)).
\]
Then there exists $C_n<\infty$, depending only on $n$, such that for every  
$R>0$, there holds 
\begin{align}\label{quantitative rate}
    \frac1R d_{\mathrm{GH}} \bigl(B_p(R),B^{\mathbb{R}^n}_0(R)\bigr) \le C_n\bigl(H_pR^{n-2}\bigr)^{\alpha_n},
    \qquad \alpha_n=[88(n+1)(4n+1)]^{-1}.
\end{align}
Here $B_p(R)$ and $B^{\mathbb{R}^n}_0(R)$ are balls in $M$ and $\mathbb{R}^n$ with radius $R$ centered at $p$ and $0$ respectively, and  $d_{\mathrm{GH}}\bigl(B_p(R),B^{\mathbb{R}^n}_0(R)\bigr)$ is the pointed Gromov-Hausdorff distance. 
\end{theorem}
Denote  $T_{p,r}M=(M,p,r^{-2}g)$.  Then,  the pointed Gromov-Hausdorff limit of $T_{p,r}M$ as $r\to 0$ is exactly $(T_pM,0,g_p)\cong(\mathbb{R}^n,0,g_{\mathrm{euc}})$, the tangent space of $M$ at $p$.  Denote the unit geodesic balls of $r^{-2}g$ and  $g_p$   centered at $p$ and $0$ by  
$$B^r_p(1)=\{x\in M| d_{r^{-2}g}(x,p)\le 1\} \quad \text{ and }\quad  B_0(1)=\{v\in T_pM| g_p(v,v)\le 1\}$$
respectively.  Then,  \eqref{quantitative rate} is equivalent to  
\begin{align*}
d_{\mathrm{GH}}\left(B^r_p(1), B_0(1)\right)\le C_n\left(H_pr^{n-2}\right)^{\alpha_n},
\end{align*}
which means the unit ball of $T_{p,r}M$ converges to the unit ball of $T_pM$ at a polynomial rate. By the experience from \cite{Colding2012}\cite{CM14}, to get such a polynomial convergence rate, one usually need to establish a quantitative metric cone estimate
\begin{align}\label{quantitative metric cone}
\Theta_r(p):=\inf_{C(Y)}\frac{1}{r}d_{\mathrm{GH}}(B_p(r),B_o^{C(Y)}(r))\le C_n \left(H_pr^{n-2}\right)^{\alpha_n},
\end{align}
and then adding this estimate in different scales $r=2^{-i}$. 
As mentioned above,  the " almost volume cone implies almost metric cone" theorem \cite{CheegerColding1996}\cite[Theorem 9.45]{Cheeger01} says, $\Theta_r(p)\le \Psi(\delta|n)$  when $B_p(r)$ satisfies the $\delta$-almost volume cone condition $(1-\delta) \mathrm{Vol}(B_p(R))\le \frac{R}{n}\mathrm{Vol}(\partial B_p(R))$.  Here $\Psi(\delta|n)$ is a function satisfying $\lim_{\delta\to 0}\Psi(\delta|n)=0$. The proof consists of two parts.  Firstly,  solve a function  $F$ satisfying the Poisson equation  
\begin{align}\label{Poisson}
\Delta F=n  \quad \text{ in } B_p(r),  \quad \quad  F|_{\partial B_p(r)}=\frac{r^2}{2},
\end{align}
and use  the $\delta$-almost volume cone condition to show $F$ satisfies  the $\delta'$ conical condition 
\begin{align*}
\|F-\frac{r^2}{2}\|_{L^\infty(B_p({(1-\delta')r})}+\fint_{B_p(r)}|\nabla F-r\nabla r|dV+\fint_{B_p(r)}|\mathrm{Hess}F-g|dV\le \delta'. 
\end{align*}
Secondly, argue that the existence of $\delta'$-conical function implies the $\Psi$-almost metric cone property. Following this strategy, to get \eqref{quantitative metric cone}, the key is to improve the $\Psi$-almost metric cone estimate to a quantitative one. Moreover, we also need to replace the almost volume cone  condition by the smallness of the  defect $H_p$ to get an almost conical function.  In our setting, instead of solving the  equation \eqref{Poisson} with Dirichlet boundary, the function $F=\frac{b^2}{2}$ does not depends on the radius $r$ and   satisfies the equation 
$$\Delta F=n|\nabla b|^2 \qquad \text{ and }\qquad F(p)=0,$$
hence is a natural choice. This relates to Colding and Minicozzi's proof of the uniqueness of tangent cone of Einstein manifolds. In \cite{CM14}, Colding and Minicozzi proved that  for Ricci-flat manifolds with Euclidean volume growth, if one tangent cone at infinity has smooth cross section,  denoting
$$b_\infty=\left(\frac{V_M}{\mathrm{Vol(B_0^{\mathbb{R}^n}(1))}}\right)^{\frac{1}{n-2}}, \qquad \text{ where } V_M=\lim_{r\to \infty}r^{-n}\mathrm{Vol}(B_p(r))>0,$$
then for $A_r:=B_p(\frac{4r}{b_{\infty}})\backslash B_p(\frac{r}{b_\infty})$ and $A^{C(Y)}_r:=B_o^{C(Y)}(\frac{4r}{b_\infty})\backslash B_o^{C(Y)}(\frac{r}{b_\infty})$, there holds 
\begin{align*}
\hat{\Theta}_{r}:=\inf_{C(Y)}\left(\frac{r}{b_\infty}\right)^{-1}d_{\mathrm{GH}}\left(A_r, A^{C(Y)}_r\right)
\le C\left[\int_{\frac{r}{2}\le b\le 8r}b^{-n}\left|\mathrm{Hess}_{b^2}-\frac{\Delta b^2}{n}g\right|^2\dd V\right]^{\frac{1}{2}}. 
\end{align*}
For general Ricci non-negative manifolds with Euclidean volume growth, they also sketch a proof of the above estimate for $\frac{1}{2}$ replaced by $\frac{1}{2+\mu}$ for some $\mu=\mu(n)$. For application in our setting, we use a similar version of this quantitative estimate under no a priori volume non-collapsing assumption: If $B_p(4)\subset M$ has non-negative Ricci curvature and  there exists a function $F\in C^2(B_p(4))$ satisfying the almost conical estimates 
$$\fint_{B_p(\frac{3}{2})}|\nabla F-r\nabla r|+|\mathrm{Hess}_F-g|\dd V\le \delta_n\ll 1,$$
then  
\begin{align*}
 \Theta_{\frac{1}{10}}(r)
    \le C_n\left(\fint_{B_p(\frac{3}{2})}|\nabla F-r\nabla r|+|\mathrm{Hess}_F-g|\dd V\right)^{\frac{1}{44(4n+1)}}.
\end{align*}
This follows from from the method of \cite{CheegerColding1996}\cite{Jiang}. See  Theorem \ref{thm:three} for a detailed chasing of the proof. To prove the function $F=\frac{b^2}{2}$ satisfies the quantitative conical estimates, we use two pieces of  global information implicit in the regular Green constant $H_p$.  The one comes from the precise remainder term of the Bishop-Gromov  relative volume comparison(monotonicity) theorem and the remainder term $h:=G-r^{2-n}$ of the Green function comparison  $G\ge r^{2-n}$. Both of the two remainder terms can be represented by a Radon measure corresponding to the remainder term of the  Laplace comparison theorem and they combines to give the quantitative Bishop--Gromov  estimate
\begin{equation}\label{eq:intro-volume-deficit}
    0\le 1-\frac{\Vol(B_p(R))}{\omega_nR^n}
    \le \frac{n}{2(n-1)}H_pR^{n-2}.
\end{equation}
Thus the scaling invariant quantity $\varepsilon_R:=H_pR^{n-2}$
controls the volume defect at scale $R$.  The other one comes from another monotonicity formula. By taking limit in  Colding's third monotonicity formula \eqref{eq:colding-third-monotonicity} and using \eqref{eq:intro_green_expansion}, we prove the global  representation
\begin{equation}\label{eq:intro-global-representation}
    H_p
    =\frac{1}{4(n-1)\sigma_{n-1}}
    \int_M D\,b^{2-2n}\dd V.
\end{equation}
These two formulae means  $H_p$  simultaneously controls the volume defect and the tensor defect $D$ from the cone equation.  Combining this two defects estimate \eqref{eq:intro-volume-deficit} \eqref{eq:intro-global-representation}  replacing the equation  \eqref{Poisson} by the equation
$$\Delta h=-w\le 0 \qquad \text{ and } \qquad h(p)=H_p,$$
when $\varepsilon_{8R}\ll 1$, we prove the quantitative conical estimate of $F$:
\begin{align}\label{eq:intro-conical-function-estimates}
   \left(\frac{1}{R^2}\|F-F_0\|_{L^\infty(B_p(R))}\right)^{\frac{1}{2}}+
    \frac{1}{R} \fint_{B_p(R)}
    |\nabla F-\nabla F_0|\dd V+\left( \fint_{B_p(R)}
    |\operatorname{Hess}F-g|^2\dd V\right)^{\frac{1}{2}}
    \leq
    C_n\varepsilon_{8R}^{\frac{1}{2(n+1)}},
\end{align}
which gives the  quantitative metric cone estimate \eqref{quantitative metric cone}. Summing the estimates in adjacent scales gives Theorem~\ref{main}.

\medskip
\noindent\textbf{Organization of the paper.}
Section~\ref{sec:Preliminary} prepares some preliminaries, including the notations, the remainder term of Bishop-Gromov volume comparison theorem and the Green representation of the  remainder term $h:=G-r^{2-n}\ge 0$ of  the Green function comparison.  Section~\ref{sec:The_Equivalence_of_Curvature_Condition_and_Green_Expansion} provides a sufficient and necessary curvature vanishing condition for the regular Green function asymptotic expansion \eqref{eq:intro_green_expansion} and  derives the explicit volume defect estimate \eqref{eq:intro-volume-deficit}.
Section~\ref{sec:Global_representation_of_H_p}  proves the global representation \eqref{eq:intro-global-representation} of $H_p$.  Section~\ref{sec:epsilon_conical_function}  proves the function $\frac{b^2}{2}$ satisfies the quantitative conical estimates \eqref{eq:intro-conical-function-estimates}.  Section~\ref{sec:Quantitative Metric Cone via conical function} provides a quantitative metric cone construction from the existence of quantitative conical functions.  Section~\ref{sec:Proof of the Main Theorem} proves the main Theorem~\ref{main}. Appendix~\ref{app:auxiliary} contains the auxiliary curvature vanishing lemmas used in Section~\ref{sec:The_Equivalence_of_Curvature_Condition_and_Green_Expansion}.

Thus the regular Green constant $H_p$ is not merely the constant term in the pole expansion.  The scale-covariant defect $H_pR^{n-2}$ controls, with explicit powers, the volume defect, the conical-function error, the metric-cone approximation, and finally the Gromov--Hausdorff distance to Euclidean space.

\textbf{Acknowledgments.} The questions and statements of each proposition are designed by the authors. We discuss with ChatGPT for some of the proofs. The manuscript was written by the authors, who take full responsibility for its contents.

\section{Preliminaries}\label{sec:Preliminary}

In this section, we prepare some notations and basic properties for complete non-compact manifolds $(M^n,g)$ with non-negative Ricci curvature. 
\subsection*{Notations}
Fix a point $p\in M$,  
$\theta\in \mathbb{S}^{n-1}\subset T_pM$ and  $t\in [-\epsilon, \epsilon]$ for $\epsilon$ sufficiently small. 
Denote
$
r=d(p,\cdot)$, $ 
    \gamma_\theta(t)=\exp_p(t\theta)
$
and let 
$
P_t:T_pM\longrightarrow T_{\gamma_\theta(t)}M
$ be the parallel transport along $\gamma_\theta$. 
Define  the radical curvature operator 
$
R_\theta(t):\theta^\perp\longrightarrow\theta^\perp
$
by
\begin{equation*}
    R_\theta(t)X
    =
    P_t^{-1}\!\left(
        \operatorname{Rm}_{\gamma_\theta(t)}
        \bigl(P_tX,\dot{\gamma}_\theta(t)\bigr)
        \dot{\gamma}_\theta(t)
    \right),
    \qquad X\in\theta^\perp.
\end{equation*}
Then, 
$
    \operatorname{tr}R_\theta(t)=
    \operatorname{Ric}(\partial_r,\partial_r)(t,\theta),
$
and
\begin{align}
R_\theta(t)
&=
\sum_{k=0}^{N}
\frac{t^k}{k!}
\bigl(\nabla_\theta^k \operatorname{Rm}\bigr)_p(\,\cdot\,,\theta)\theta
+
O(t^{N+1}),\label{eq:Taylor-R-theta}\\
\operatorname{Ric}(\partial_r,\partial_r)(t,\theta)
&=
\sum_{k=0}^{\ell}
\frac{t^k}{k!}
\bigl(\nabla_\theta^k \operatorname{Ric}\bigr)_p(\theta,\theta)
+
O(t^{\ell+1}),\label{eq:Taylor-Ric-theta}
\end{align}
where $\bigl(\nabla_\theta^k \operatorname{Rm}\bigr)_p(\,\cdot\,,\theta)\theta$ and $\bigl(\nabla_\theta^k \operatorname{Ric}\bigr)_p(\theta,\theta)$  are all smooth on $\theta$. For $v\in\theta^\perp$, let $Y_v(t)$ be the  Jacobi field
along $\gamma_\theta$ satisfying
$
    Y_v(0)=0,
    \nabla_{\dot{\gamma}_\theta}Y_v(0)=v.
$
Define
$$
    J_\theta(t)v:=P_t^{-1}\bigl(Y_v(t)\bigr),\qquad 
K_\theta(t)=t^{-1}J_\theta(t)
\qquad \text{ and }  \qquad J(t,\theta):=\det J_\theta(t).$$ 
Then $K_\theta(0)=\operatorname{Id}$, and for  small $t>0$,  $K_\theta(t)$ is invertible,  the volume element $\dd V$ satisfies 
\begin{align}\label{volume estimate}
\dd V=J(t,\theta)dtd\theta=t^{n-1}j(t,\theta)dtd\theta \qquad \text{ for some  } \qquad 
j(t,\theta)\in [\frac 12, 2],
\end{align}
and the shape operator $S_\theta(t)$ satisfies
\begin{equation*}
    S_\theta(t)
    =
    J_\theta'(t)J_\theta^{-1}(t)
    =
    \frac{1}{t}\operatorname{Id}
    +
    K_\theta'(t)K_\theta^{-1}(t).
\end{equation*}
Define
$
 B_\theta(t):=
    K_\theta'(t)K_\theta^{-1}(t)$ and $     \beta(t,\theta):=\operatorname{tr}B_\theta(t).
$
Then
\begin{equation}
    \Delta r
    =
    \operatorname{tr}S_\theta(t)
    =
    \frac{n-1}{t}+\beta(t,\theta).
    \label{eq:laplacian-r-beta}
\end{equation}
For fixed $\theta$, the above defined functions  $R_\theta(t), \Ric(\partial r, \partial r)(t,\theta), J_\theta(t), K_\theta(t), B_\theta(t), \beta(t,\theta)$  can  all  be extended to  smooth functions  on $[-\epsilon,\epsilon]$\cite[Section 4]{HuangXiongPole2026}.  By \cite[Lemma 4.1, Lemma 4.3]{HuangXiongPole2026}, the  Riccati equation 
\begin{equation}\label{eq:Riccati-B-theta}
B_\theta'(t)
+
\frac{2}{t}B_\theta(t)
+
B_\theta(t)^2
+
R_\theta(t)
=
0.
\end{equation}
implies
\begin{equation}
    \beta(t,\theta)
    =
    -t^{-2}
    \int_0^t
    s^2
    \left(
        |B_\theta(s)|^2
        +
        \operatorname{Ric}(\partial_r,\partial_r)(s, \theta)
    \right)\dd s, \quad |B_\theta(t)|^2
:=
\operatorname{tr}\bigl(B_\theta(t)^2\bigr).
    \label{eq:beta-integral}
\end{equation}
and 
\begin{equation}
    \Delta r
    =
    \frac{n-1}{r}
    -
    \frac{1}{3}\operatorname{Ric}_p(\theta,\theta)\,r
    +
    O(r^2) \quad \mathrm{as} \quad r\to 0.
    \label{eq:laplacian-r-expansion}
\end{equation}

\subsection*{Monotonicity identity}
\begin{lemma}\label{prop:distribution-laplacian}
Let $M$ be a complete Riemannian manifold with
$\operatorname{Ric}\geq0$. Then there exists a  Radon measure $w$ on $M$ such that, in the sense of distributions,
\begin{equation}
    \Delta r^{2-n}
    =
    -(n-2)\sigma_{n-1}\delta_p+w,
    \label{eq:distribution-laplacian}
\end{equation}
and $
w|_{B_p(t)} \in L^1(B_p(t))$ for $ t< \operatorname{inj}_p(M).
$
\end{lemma}

\begin{proof}
For any test function $\varphi\in C_c^\infty(M)$ and $t\le \mathrm{inj}_p(M)$, by partition of unit, there exists $\phi=\lambda\varphi\in C_c^\infty(B_p(t))$ and $\eta=(1-\lambda)\varphi\in C_c^\infty(M\backslash \{p\})$ such that $\varphi=\phi+\eta$. Thus 
\begin{align*}
\int_{M}\Delta{r}^{2-n} \varphi \dd V:=\int_{M}r^{2-n}\Delta \varphi \dd V=\int_{B_p(t)}r^{2-n}\Delta \phi+\int_{M\backslash \{p\}}r^{2-n}\Delta \eta.
\end{align*}
Let
$
A_\varepsilon
=
B_p(t)\setminus\overline{B_p(\varepsilon)}.
$
Then,
\begin{align*}
\int_{A_\varepsilon}r^{2-n}\Delta\phi\,\dd V
    =
    \int_{A_\varepsilon}
    \phi\,\Delta r^{2-n}\dd V
    \quad
    +
    \int_{S_p(\varepsilon)}
    \left(
        r^{2-n}\partial_\eta\phi
        -
        \phi\,\partial_\eta r^{2-n}
    \right)\dd A.
\end{align*}
On the inner boundary,
\[
\left|
\int_{S_p(\varepsilon)}
r^{2-n}\partial_\eta\phi\dd A
\right|
\leq
C\varepsilon^{2-n}
\operatorname{Area}(S_p(\varepsilon))
=
O(\varepsilon) \to 0, \quad \eps \to 0, 
\]
and
\[
-\int_{S_p(\varepsilon)}
\phi\,\partial_\eta r^{2-n}\dd A
=
(2-n)\varepsilon^{1-n}
\int_{S_p(\varepsilon)}\phi\dd A
\longrightarrow
(2-n)\sigma_{n-1}\phi(p), \quad \eps \to 0.
\]
Moreover, by \eqref{eq:laplacian-r-expansion}, we know $$w_0:=\Delta r^{2-n}|_{B_p(t)\setminus \{0\}}=(n-2)r^{1-n}
\left(
    \frac{n-1}{r}-\Delta r
\right)=O(r^{2-n}).$$
Thus $w_0$ extends to a function in  $L^1(B_p(t))$ and 
\[
\int_{A_\varepsilon}
\phi\,\Delta r^{2-n}\dd V=\int_{A_\varepsilon}\phi w_0\dd V
\longrightarrow
\int_{B_p(t)}\phi w_0\dd V, \quad \eps \to 0.
\]
So, 
\[
\int_{B_p(t)}r^{2-n}\Delta\phi\dd V
=
\int_{B_p(t)}\phi w_0\dd V
+
(2-n)\sigma_{n-1}\phi(p).
\]
For the second term, since $r^{2-n}$ is a Lipschitz function on a neighborhood of $\mathrm{spt}\eta\subset\subset M\backslash\{p\}$, we know 
\begin{align*}
\int r^{2-n}\Delta\eta\dd V=-(2-n)\int \nabla r\cdot \nabla(r^{1-n}\eta)\dd V+(2-n)(1-n)\int r^{-n}\eta\dd V
\end{align*}
Since $\operatorname{Ric}\geq0$, the Laplacian comparison theorem\cite[Theorem 4.6]{Cheeger01} implies,
in the distributional sense,
$$\Delta r\le \frac{n-1}{r}.$$
Since $\frac{n-1}{r}\in L^1_{loc}(M)$, we know  
$w_1:=\frac{n-1}{r}-\Delta r$ is a Radon measure in the sense of distribution and $w_0 \dd V=(n-2)r^{1-n}w_1|_{B_p(t)}\ge 0$. Thus  
\begin{align*}
\int r^{2-n}\Delta\eta\dd V=(2-n)\int \Delta r\cdot r^{1-n}\eta \dd V+(2-n)(1-n)\int r^{-n}\eta\dd V=(n-2)\int \eta r^{1-n}\dd w_1.
\end{align*}
As a result, 
\begin{align*}
\int_{M}\Delta{r}^{2-n} \varphi \dd V=\int_{B_p(t)}\phi w_0\dd V
+
(2-n)\sigma_{n-1}\phi(p)+(n-2)\int \eta r^{1-n} \dd w_1
\end{align*}
Letting $w=\lambda w_0\dd V+(n-2)(1-\lambda) r^{1-n} w_1=(n-2)r^{1-n}w_1$, we know 
\[
\Delta r^{2-n}
=
-(n-2)\sigma_{n-1}\delta_p+w,
\]
$w\ge0$ and  $
w|_{B_p(t)} \in L^1(B_p(t))$ for $t< \operatorname{inj}_p(M).$
\end{proof}
We now follow Leon Simon's technique\cite{Simon2018} to prove the Bishop-Gromov volume comparison theorem with precise remainder term, which will be used to deduce the volume non-collapsing condition from the  defect $H_p$.    
\begin{lemma}\label{monotonicity identity}
   Let $M$ be a complete Riemannian manifold with
$\operatorname{Ric}\geq0$. Then, for any $0<\rho_1<\rho_2<\infty$, there holds 
\begin{align*}
    \frac{V(\rho_2)}{\rho_2^n}-\frac{V(\rho_1)}{\rho_1^n}=-\int_{\rho_1}^{\rho_2}\frac{\int_{B_p(\rho)}rdw_1}{\rho^{n+1}}\dd\rho,
    \end{align*}
where $V(\rho)=\mathrm{Vol}(B_p(\rho))$ and $w_1=\frac{n-1}{r}-\Delta r$ is a Radon measure. 
\end{lemma}
\begin{proof}
    Suppose $\varphi_\eps: [0,\infty) \to [0,1] \subset \R$ is a smooth function that satisfies 
    \begin{equation*}
        \varphi_\eps(t) = \left\{\begin{matrix}
            0 , \; t \geq 1 + \eps \\
            1 , \; t \leq 1.
        \end{matrix} \right., \qquad \varphi_{\eps}'(t) \leq 0 \text{ for all } t    \end{equation*}  
    Since $\varphi_{\eps}(r(x)/\rho)$ is compact supported Lipschitz function, we know  \begin{equation*}
            \langle \Delta r^{2}, \varphi_\eps \rangle = \int_M r^{2} \Delta \varphi_\eps(\frac{r}{\rho})\dd V= -\int_{M} \grad r^{2} \cdot \grad \varphi_\eps(\frac{r}{\rho})  \dd V= -2\int_{M}  \frac{r}{\rho}\varphi'_\eps(\frac{r}{\rho}) \dd V.
    \end{equation*} 
By
    $
        \varphi'_\eps(\frac{r}{\rho})(\frac{r}{\rho}) = - \rho\frac{d}{d\rho} (\varphi(\frac{r}{\rho})),
    $
    we get
    \begin{equation*}
        \langle \Delta r^{2}, \varphi_\eps \rangle = 2\rho  \frac{d}{d\rho}\int_{M}  (\varphi_\varepsilon(\frac{r}{\rho})) \dd V 
    \end{equation*}
Divide both sides by $2\rho$ and take integration  in $[\rho_0, \rho_1]$, we get 
    \begin{equation*}
    \begin{aligned}
       \int_{\rho_0}^{\rho_1} \frac{1}{2\rho} \langle \Delta r^{2}, \varphi_{\eps} \rangle \dd \rho = \int_M \varphi_\varepsilon(\frac{r}{\rho_1}) - \varphi_\varepsilon(\frac{r}{\rho_0}) \dd V(x)
    \end{aligned}
    \end{equation*}
     Let $\eps \to 0$, and divide both sides by $\rho_1 - \rho_0$,
    \begin{equation*}
        \begin{aligned}
            \frac{1}{2(\rho_1 - \rho_0)}\int_{\rho_0}^{\rho_1}\rho^{-1} \Delta r^{2}(B_p(\rho))\dd \rho&= \frac{V(\rho_1)-V(\rho_0)}{\rho_1 - \rho_0}. 
        \end{aligned}
    \end{equation*}
    Letting $\rho_1 \to \rho_0$, we get 
    \begin{equation*}
        \frac{1}{2\rho}\Delta r^{2}(B_p(\rho)) =  V'(\rho) \quad \text{ for a.e.} \quad \rho>0,
    \end{equation*}
    Since  $w_1=\frac{n-1}{r}-\Delta r$ is a Radon measure, we know $\Delta r=2(n-r w_1)$in the distribution sense, hence 
    \begin{align*}
    V'=\frac{nV(\rho)-\int_{B_\rho}r\dd w_1}{\rho}.
    \end{align*}  This further implies 
    \begin{align*}
    \frac{d}{d\rho}\left(\frac{V}{\rho^n}\right)=-\frac{\int_{B_p(\rho)}r\dd w_1}{\rho^{n+1}},
    \end{align*}
    and 
    \begin{align*}
    \frac{V(\rho_2)}{\rho_2^n}-\frac{V(\rho_1)}{\rho_1^n}=-\int_{\rho_1}^{\rho_2}\frac{\int_{B_p(\rho)}rdw_1}{\rho^{n+1}}\dd \rho.
    \end{align*} 
\end{proof}
\begin{corollary} \label{lem:Riesz_Measure_Monotonicity}
 Let $M$ be a complete Riemannian manifold with
$\operatorname{Ric}\geq0$ and dimension at least three and $w$ be the Radon measure introduced in Lemma \ref{prop:distribution-laplacian}.  Then,  for every $R>0$,
    \begin{equation}\label{eq:volume-riesz-exact}
        1 - \frac{V(R)}{\omega_n R^n} = \frac{1}{(n-2)\omega_n R^n} \int_0^R \rho^{n-1} w (B_p(\rho)) \dd \rho.
    \end{equation}
\end{corollary}
\begin{proof}
 In a neighborhood of  $\operatorname{spt}(\grad \varphi_{\eps}(r(x)/\rho)) \subset B_p((1+\eps)\rho)\backslash B_p(\rho)$,  $r^{2-n}$ is  Lipschitz   with  $\grad r^{2-n} = (2-n)r^{1-n} \grad r $. By similar calculation as in the proof of Lemma \ref{monotonicity identity}, we have 
 \begin{equation*}
        \langle \Delta r^{2-n}, \varphi_\eps \rangle = (2-n)\int_{M} \rho r^{-n} \frac{d}{d\rho} (\varphi_\varepsilon(\frac{r}{\rho})) \dd V. 
    \end{equation*}
Notice that $\varphi'(r/\rho)$ is nonzero only when $ 1 \leq r\rho^{-1}\leq 1+\eps $. Thus 
    \begin{align*}
        (1+\eps)^{-n} \frac{d}{d\rho} \int_{M} (\varphi_\varepsilon(\frac{r}{\rho})) \dd V \leq \frac{1}{2-n}\rho^{n-1} \langle \Delta r^{2-n}, \varphi_\eps \rangle \leq \frac{d}{d\rho} \int_{M} (\varphi_\varepsilon(\frac{r}{\rho})) \dd V. 
    \end{align*}
    Integrating on   $[\rho_0, \rho_1]$, letting $\varepsilon\to 0$ and then dividing both sides by $\rho_1-\rho_0$, we get 
    \begin{equation*}
        \begin{aligned}
            \frac{1}{(2-n)(\rho_1 - \rho_0)}\int_{\rho_0}^{\rho_1}\rho^{n-1} \Delta r^{2-n}(B_p(\rho))\dd \rho= \frac{V(\rho_1) - V(\rho_0)}{\rho_1-\rho_0 }.    
        \end{aligned}
    \end{equation*}
    Let $\rho_1 \to \rho_0$ and apply \eqref{eq:distribution-laplacian},  we get 
    \begin{equation*}
       V'(\rho)= \frac{\rho^{n-1}}{2-n}\Delta r^{2-n}(B_p(\rho))=\frac{\rho^{n-1}w(B_p(\rho))}{2-n}+\sigma_{n-1}\rho^{n-1},
    \end{equation*}
   Integrating on $[0,R]$, we get  
   \begin{equation*}
    \omega_nR^n- V(R)
    =\frac1{n-2}\int_0^R \rho^{n-1} w (B_p(\rho)) \dd \rho,
    \end{equation*}
    which implies \eqref{eq:volume-riesz-exact}
\end{proof}
\subsection*{Remainder term of Green's function comparison}
Now, assume  additional  $(M^{n\ge 3},g)$ is  non-parabolic. Let $G_p=G(p,\cdot)$ be the normalized minimal positive Green function. By the maximal principle(see \cite[Lemma 2.15]{Colding2012} for example), there holds 
$G_p\ge r^{2-n}.$
The following lemma gives a representation  of the error 
$$h(x):=G_p(x)-r^{2-n}.$$
\begin{lemma}\label{prop:green-riesz-representation}
Let  $M^{n\ge 3}$ be a non-parabolic manifold with $\mathrm{Ric}\ge 0$ and $h$ defined above. Then,
\begin{equation}
    h(x)
    =
    \frac{1}{(n-2)\sigma_{n-1}}
    \int_M G(x,y)\dd w(y),\forall x\in M\backslash \{p\},
    \label{eq:green-riesz-global}
\end{equation}
where $w$ is the  Radon measure in Lemma~\ref{prop:distribution-laplacian}.
\end{lemma}
\begin{proof}
Let
$
    \Omega_1 \subset \subset \Omega_2 \subset \subset \cdots \subset \subset M,
$
be a smooth compact exhaustion of $M$ and  $G_{\Omega_j}(x,y)$ be  the Dirichlet Green function on
$\Omega_j$. That is, 
\[
    -\Delta_y G_{\Omega_j}(x,y)
    =
    c_n\delta_x,
    \qquad
    c_n=(n-2)\sigma_{n-1},
\qquad 
    G_{\Omega_j}(x,\cdot)=0
    \text{ on }\partial\Omega_j.
\]
Then, by \cite{LiTam1987},  for every fixed $x\neq y$ and all  $j\gg 1$,
$
    G_{\Omega_j}(x,y)
    \nearrow
    G(x,y).
$
 Let 
    \begin{equation*}
        u_j(x) := \frac{1}{(n-2)\sigma_{n-1}} \int_{\Omega_j} G_{\Omega_j}(x,y) \dd w(y) .
    \end{equation*}
 By the symmetry of the Green function, we know 
  $ -\Delta u_j = w|_{\Omega_j}  $ and $ 
            u_j |_{\partial \Omega_j} = 0.
$
    Thus 
    \begin{equation*}
            -\Delta (h -u_j) = 0 , \quad \text{in } \Omega_j \qquad \text{ and } \quad 
            h-u_j \geq 0, \quad \text{on } \partial \Omega_j.
    \end{equation*}
    Since $u_j$, $h \in L^1_{\mathrm{loc}}(M)$, by Weyl lemma and maximum principle, $h\ge u_j \geq 0$ on $\Omega_j$.
    Since
$
    G_k(\cdot,x) \nearrow G(\cdot,x),
$
we have $u_k$ converge monotonically to some $u\ge 0$.
By the monotone convergence theorem,
\begin{equation*}
    0\le u(x)
    =
    \frac{1}{(n-2)\sigma_{n-1}}
    \int_M G(x,y) \dd w(y) 
    =
    \lim_{k\to\infty}u_k(x)\le h(x).
\end{equation*}
Since $h \in L_{\mathrm{loc}}^1(M)$, we know  $u\in L_{\mathrm{loc}}^1(M)$. For any  $\varphi\in C_c^\infty(M)$ and $j\gg 1$, there holds 
$\operatorname{supp}\varphi\subset \subset \Omega_j$. So, 
$
    |u_j(-\Delta\varphi)|
    \leq
    h|\Delta\varphi| \in L^1(M).
$
Thus by the dominated convergence theorem and Fubini's theorem, we know

\begin{equation*}
    \int_M u(-\Delta\varphi)\dd V
    =\lim_{k\to \infty}\int_M u_j(-\Delta \varphi) \dd V=\lim_{k\to\infty}\int_{\Omega_k}\varphi(y)\dd w=
    \int_M \varphi\dd w.
\end{equation*}  
Thus
$
    -\Delta u=w=-\Delta h
$
in the distributional sense. It follows that $h-u$
is a nonnegative harmonic function on $M$ in the distributional sense. By Weyl's lemma, $h-u$ is smooth.
By the Green function estimates of Li and Yau\cite[Theorem 5.2]{LiYau1986}
(see also \cite[Theorem 3.1]{LiTam1995}
for a purely elliptic proof), we have
$
    G(p,x)\to0
    \text{ as } x\to\infty.
$
Noting, 
\[
    0\leq h(x)
    =
    G(p,x)-r(x)^{2-n}
    \leq
    G(p,x),
\]
we have
$
    h(x)\to 0
    $,  hence $(h-u)(x)\to 0 $  as  $x\to\infty$. 
By the maximal principle,
$
    h-u\equiv0.
$
Thus 
\[
    h(x)=u(x)=\frac{1}{(n-2)\sigma_{n-1}}
    \int_M G(x,y) \dd w(y) 
\].

\end{proof}

\section{The Equivalence of Curvature Condition and volume defect estimate via $H_p$}\label{sec:The_Equivalence_of_Curvature_Condition_and_Green_Expansion}
In \cite{HuangXiongPole2026}, it's proved some   curvature vanishing conditions are sufficient  for the expansion of the minimal  Green function. 
\begin{proposition}[Huang--Xiong \cite{HuangXiongPole2026}]
\label{prop:huang-xiong-pole-expansion} Let $(M^{n\ge 3},g)$ be a non-parabolic manifold and $p\in M$. Assume  
 $\operatorname{Ric}\geq0$  in a domain $\Omega$ containing $ p$. Suppose that
\begin{equation}
\begin{aligned}
    \nabla^k\operatorname{Ric}(p)&=0,
    \qquad
    &&0\leq k\leq n-4,
    \\
    \nabla^l\operatorname{Rm}(p)
    &=
    0,
    &&0\leq \ell\leq
\left\lfloor\frac{n-6}{2}\right\rfloor.
\end{aligned}
\label{eq:huang-xiong-jet-conditions}
\end{equation}
Then for any local solution $u$ of $-\Delta u=(n-2)\sigma_{n-1}\delta_p$ in $\Omega$, there holds  

\begin{equation*}
    u(x) =r^{2-n}+H+O_1(r),
\end{equation*}
where $H$ is a constant depending on $u$ and $O_1(r)$ is a Lipschitz function which vanishes at $p$.  Moreover, if $M$ is non-parabolic, then the minimal normalized Green function $G_p$ satisfies \eqref{eq:intro_green_expansion}
and $H_p\ge 0$. 
We call  $H_p$ the regular Green constant at $p$. 
\end{proposition}
In this  section, we will show the conditions \eqref{eq:huang-xiong-jet-conditions} are also necessary for the expansion \eqref{eq:intro_green_expansion}, and prove  $H_p$ provides  an upper  bound of   $w$ and  a lower bound of the  volume ratio.  

\begin{proposition}\label{prop:curvature-jet-vanishing}
Let $(M^{n\ge 3},g)$ be a non-parabolic manifold and with  $\operatorname{Ric}\geq0$, $p\in M$ and $G_p$ be the minimal normalized  Green function. 
Suppose
\[
    H_p:=\lim_{x\to p}
    \bigl(G_p(x)-r^{2-n}\bigr) \text{ exists}. 
\]
Then \eqref{eq:huang-xiong-jet-conditions} holds. 
\end{proposition}

\begin{proof} Let $h(x)=G_p(x)-r^{2-n}$, by  \eqref{eq:green-riesz-global} and Fatou's Lemma, we know 
\[
    H_p=\lim_{x_i\to p}h(x_i)=\lim_{x_i\to p}\frac{1}{(n-2)\sigma_{n-1}}\int_{M}G(x_i,y)\dd w(y)
    \ge 
    \frac{1}{(n-2)\sigma_{n-1}}
    \int_M G(p,y)\,\dd w(y).
\]
Since $G_p\ge r^{2-n}$, we know that for any $R>0$,  
\begin{equation}\label{eq:weighted-omega-bound}
    \int_{B_p(R)} r^{2-n}\,\dd w
    \leq (n-2)\sigma_{n-1} H_p.
\end{equation}
For $\rho\ll \mathrm{inj}_p(M)$, by Lemma \ref{prop:distribution-laplacian},  \eqref{eq:laplacian-r-beta} and \eqref{volume estimate},   we know that  on $B_p(\rho)\backslash \{p\}$,  there holds 
\[
\begin{aligned}
   \dd w&=\Delta r^{2-n} \dd V=
    (2-n)r^{1-n}
    \left(
        \Delta r-\frac{n-1}{r}
    \right)\dd V\\
    &=
    (2-n)r^{1-n}\beta(r,\theta)r^{n-1}j(r,\theta)\dd r\dd\theta\\
    &=(2-n)\beta(r,\theta) j(r,\theta)\dd r \dd\theta
\end{aligned}
\]
where $\beta(r,\theta)\leq0$ and $j(r,\theta)\in [\frac 12, 2]$.
By \eqref{eq:weighted-omega-bound},
\begin{equation}
\begin{aligned}
    \int_0^\rho
    r^{2-n}
    \int_{S_pM}
    (n-2)
    \bigl(-\beta(r,\theta)\bigr)
    \dd\theta\dd r\le 2\int_{B_p(\rho)} r^{2-n}\dd w
    &
    <+\infty.
\end{aligned}
\label{eq:weighted-omega-integrability}
\end{equation}
Since  $\beta(t,\theta)$ is smooth in $t\in [-\epsilon, \epsilon]$ we know 
\begin{align*}
\beta(t,\theta)=\sum_{l=0}^{k}a_l(\theta)t^l+O(t^{k+1}),
\end{align*}
where $a_l(\theta)$ depends on $\theta$ smoothly\cite[Lemma 4.2]{HuangXiongPole2026}. 
Suppose there exists $q\ge 0$ such that $a_q\ne0$,   then there exists $K\subset S_pM$ such that  $\mathrm{Area}(K)>0$ and $|a_q(\theta)|\ge \delta_0>0$, without loss of generality, assuming $a_{l}\equiv 0$ for $l<q$, then 
\[
    \beta(t,\theta)
    =
    t^q a_q(\theta)+O(t^{q+1}),
    \qquad
   |a_q(\theta)|\ge \delta_0 \text{ for } \theta\in K.
\]
By \eqref{eq:weighted-omega-integrability}, one has
$q\geq n-2$, and hence $\beta(r,\theta) = O(r^{n-2})$. If there is no $q\ge 0$ such that $a_q\neq 0$, then $\beta(t,\theta) = O(t^{n-2})$ also holds.
It follows from \eqref{eq:beta-integral} that
\begin{equation}
    \operatorname{Ric}(\partial_r,\partial_r)(t,\theta)
    =
    O(t^{n-3}),
    \qquad
    |B(t,\theta)|^2
    =
    O(t^{n-3}).
    \label{eq:ricci-B-decay}
\end{equation}
By  \eqref{eq:Taylor-Ric-theta} and Proposition~\ref{prop:ricci-radial-jet}, 
\[
    \nabla^k\operatorname{Ric}(p)=0,
    \qquad
    0\leq k\leq n-4.
\]

For each fixed $\theta \in \mathbb{S}^{n-1}$, $B(t,\theta)$ is smooth in $t$, and
$B(0,\theta)=0$.
Suppose that the order of its first nonzero Taylor term is $q\geq1$:
\[
    B(r,\theta)
    =
    r^q A_q(\theta)+O(r^{q+1}),
    \qquad
    A_q(\theta)\neq0,
\]
where $A_q(\theta)$ is symmetric and smooth.
Since $S_{\theta} (t)$ is self-adjoint, $B(t,\theta)$ is also self-adjoint. Then, 
\[
\begin{aligned}
    |B(r,\theta)|^2
    = \operatorname{tr}\bigl(B(r,\theta)^2\bigr)
    =r^{2q}\operatorname{tr}\bigl(A_q(\theta)^2\bigr)
    +
    O(r^{2q+1}).
\end{aligned}
\]
Since $A_q(\theta)\neq0$ is symmetric,
$\operatorname{tr}\bigl(A_q(\theta)^2\bigr)\neq0$.
Therefore, \eqref{eq:ricci-B-decay} implies
$q\geq
    \left\lfloor\frac{n-3}{2}\right\rfloor$.
The Riccati equation \eqref{eq:Riccati-B-theta} then gives
\[
    R_\theta(t)=O(t^\ell),
    \qquad
    \ell=
    \left\lfloor\frac{n-3}{2}\right\rfloor-1.
\]
By \eqref{eq:Taylor-R-theta}, lemma~\ref{lem:jacobi-jet-isomorphism} and the arbitrariness of $\theta$,
\[
    \nabla^j\operatorname{Rm}(p)=0,
    \qquad
    0\leq j\leq
    \left\lfloor\frac{n-6}{2}\right\rfloor.
\]
\end{proof}

\begin{theorem}
\label{thm:odd-dimensional-equivalence}Let $(M^{n\ge 3},g)$ be a non-parabolic manifold and with  $\operatorname{Ric}\geq0$, $p\in M$ and $G_p$ be the minimal normalized  Green function.  Then
$$      \lim_{x\to p}
        \bigl(G(p,x)-r(x)^{2-n}\bigr)
        \text{ exists}.
   \Longleftrightarrow
        \eqref{eq:huang-xiong-jet-conditions}
    \Longleftrightarrow
        G(p,x)
        =
        r(x)^{2-n}+H_p+O_1(r(x)).
$$
\end{theorem}

\begin{proof}
The conclusion follows from  Propostion~\ref{prop:curvature-jet-vanishing}
and  Proposition~\ref{prop:huang-xiong-pole-expansion}.
\end{proof}

From now on, we will always assume $M$ is a parabolic manifold with $\mathrm{Ric}]\ge 0$ and its minimal normalized Green function has the asymptotic expansion 
\eqref{eq:intro_green_expansion}. 
\begin{corollary}\label{cor:omega-ball-bound}
Let $(M^{n\ge 3},g)$ be a non-parabolic manifold and with  $\operatorname{Ric}\geq0$, $p\in M$ and assume $G_p$ be the minimal normalized Green function satisfying \eqref{eq:intro_green_expansion}. Then for every $R>0$,
\begin{equation}
    w\bigl(B_p(R)\bigr)
    \leq
    (n-2)\sigma_{n-1}H_p R^{n-2}.
    \label{eq:omega-ball-bound}
\end{equation}
\end{corollary}

\begin{proof}
Since $G(p,x) = r^{2-n}+H_p+O_1(r)$, for any $R>0$,  by \eqref{eq:weighted-omega-bound}, we know 
\[
    R^{2-n}w\bigl(B_p(R)\bigr)\le  \int_{B_p(R)}r^{2-n}\dd w
    \leq
    (n-2)\sigma_{n-1}H_p.
\] 
\end{proof}

\begin{proposition}\label{thm:volume-ratio-estimates}
Let $(M^{n\ge 3},g)$ be a non-parabolic manifold and with  $\operatorname{Ric}\geq0$, $p\in M$ and assume $G_p$ be the minimal normalized Green function satisfying \eqref{eq:intro_green_expansion}  Then, for  every $R>0$,
\begin{equation*}
    0
    \leq
    1-\frac{\operatorname{Vol}(B_p(R))}
         {\omega_nR^n}
    \leq
    \frac{n}{2(n-1)}H_pR^{n-2}.
\end{equation*}
\end{proposition}

\begin{proof}By 
\eqref{eq:volume-riesz-exact} and 
\eqref{eq:omega-ball-bound}, we know 
\[
\begin{aligned}
    \omega_nR^n-\operatorname{Vol}(B_p(R))
    &\leq
    \sigma_{n-1}H_p
    \int_0^R \rho^{2n-3}\dd \rho
    =
    \frac{\sigma_{n-1}}{2(n-1)}
    H_pR^{2n-2}.
\end{aligned}
\]
\end{proof}

Combining this proposition with  Colding's almost rigidity theorem for Ricci non-negative manifolds with almost maximal volume \cite[Theorem~0.8]{Colding1997}, see also\cite[Theorem~9.69]{Cheeger01}, we have the following Theorem. 

\begin{theorem}\label{thm:green-function-gh-rigidity}
Let $(M^{n\ge 3},g)$ be a non-parabolic manifold and with  $\operatorname{Ric}\geq0$, $p\in M$ and assume $G_p$ be the minimal normalized Green function satisfying \eqref{eq:intro_green_expansion} 
Then, for any $R>0$, there holds 
\begin{equation*}
    \frac{1}{R}
    d_{\mathrm{GH}}
    \left(
       B_p(R),
        B_0^{\mathbb{R}^n}(R)
    \right)
    \leq
    \Psi_n\!\left(H_pR^{n-2}\right).
\end{equation*}
where $\Psi_n$ is a function satisfying $
    \lim_{t\to 0}\Psi_n(t)=0.$
\end{theorem}
The  rest of the paper is devoted to improve the above $\Psi_n$ to be  a polynomial. In \cite{XuZhou23}, quantitative splitting map is constructed under the almost maximal volume growth condition.  More precisely, For a geodesic ball $B_p(1)\subset M$ with non-negative Ricci curvature and volume $\mathrm{Vol}(B_p(1))\ge (1-\tau)\omega_n$, it is proved\cite[Proposition 3.17, Theorem 3.18]{XuZhou23} that for $r_{q_n}=\beta^{-\nu}$, $\nu\ge 4^{n+1}, \beta\le \tau^{-c(n)}$,  there exists a Gromov-Hausdorff approximation $f: B_p(r_{q_n})\to B_0(r_{q_n})\subset \mathbb{R}^n$ such that 
$$r_{q_n}^{-1}d_{GH}\left(B_p(r_{q_n}),B_0^{\mathbb{R}^n}(r_{q_n})\right)\le C\beta^{-\frac{1}{2}}.$$
In the construction, quantitative splitting theorem is used.  The radius $r_{q_n}=\beta^{-\nu}$ is very small as the error $C\beta^{-1}=C\tau^{c(n)}$ goes to zero. Instead of using the splitting theorem, the strategy of this paper is to use a quantitative "almost volume cone implies almost metric cone" property in different scales and add them together to get the polynomial power estimate in any scale.   
\section{A Global representation of $H_p$}\label{sec:Global_representation_of_H_p}
For a function $v\in C^1(M\setminus\{p\})$, Colding \cite{Colding2012} defines
\begin{equation*}
    I_v(s)
    :=
    s^{1-n}
    \int_{\{b=s\}}
    v|\nabla b|\dd A,  \qquad  A(s):=I_{|\nabla b|^2}(s),
\end{equation*}
and computes 
\begin{equation*}
        I_v'(s)
    =
    s^{1-n}
    \int_{\{b=s\}}
    v_\nu\dd A,
\end{equation*}
where $\nu$ denotes the outward unit normal to the regular level set
$\{b=s\}$. Especially, when  $v\equiv1$,  $ I_1(s)\equiv|\mathbb S^{n-1}|= \sigma_{n-1}$.  We will give a global representation of $H_p$ by \eqref{eq:colding-third-monotonicity}. For simplicity, denote 
$$q(s)=s^{3-n}A'(s).$$ 
\begin{theorem}\label{thm:q-limit-at-zero}Let $(M^{n\ge 3},g)$ be a non-parabolic manifold and with  $\operatorname{Ric}\geq0$, $p\in M$ and assume $G_p$ be the minimal normalized Green function satisfying \eqref{eq:intro_green_expansion} 
Then,   for every $R>0$,
\begin{align}\label{local representation}
    q(R)+2(n-1)\sigma_{n-1}H_p
    =
    \frac12
    \int_{\{b\leq R\}}
    D\,b^{2-2n}\dd V.
\end{align}
Moreover, we have  
\begin{equation*}
    H_p
    =
    \frac{1}{4(n-1)\sigma_{n-1}}
    \int_M D\,b^{2-2n}\dd V.
\end{equation*}
and 
\begin{equation}\label{eq:Hp-D-identity}
    R^{-n}
    \int_{\{b\leq R\}}D\dd V
    \leq
    4(n-1)\sigma_{n-1}H_p R^{n-2}.
\end{equation}

\end{theorem}

\begin{proof}
Since
    $G
    =
    r^{2-n}+H_p+O_1(r)$,
we have
\[
    b
    =
    r-\frac{H_p}{n-2}r^{n-1}
    +O_1(r^n), \quad
    |\nabla b|^2
    =
    1-\frac{2(n-1)}{n-2}H_p r^{n-2}
    +O(r^{n-1}).
\]
Since
\[
    \sigma_{n-1}
    =
    I_1(s)
    =
    s^{1-n}\int_{\{b=s\}}|\nabla b|\dd\sigma,
\]
it follows from the definition of $A$ that
\begin{equation}
\begin{aligned}
    A(s)-\sigma_{n-1}
    &=
    s^{1-n}
    \int_{\{b=s\}}
    \bigl(|\nabla b|^2-1\bigr)|\nabla b|\dd\sigma
    =
    -\frac{2(n-1)}{n-2}
    \sigma_{n-1}H_p\,s^{n-2}
    +O(s^{n-1}).
\end{aligned}
\label{eq:A-small-s-expansion}
\end{equation}
By \cite[Corollary 2.19]{Colding2012}, we know  $A'(s)\leq0$, which implies  $q(s)=s^{3-n}A'(s)\leq0$.
First, we give the rough estimate  $$\lim_{s\to0^+}q(s)>-\infty.$$
Suppose, to the contrary, that $\lim_{s\to0^+}q(s)=-\infty$.
Then for every $N>0$, there exists $s(N)>0$ such that, whenever
$0<s<s(N)$, $q(s)<-N$.
Hence the definition of $q$ gives
\[
\begin{aligned}
    A(s)-A(\lambda s)
    &=
    \int_{\lambda s}^{s}t^{n-3}q(t)\dd t
    \leq
    -\frac{N}{n-2}
    \bigl(1-\lambda^{n-2}\bigr)s^{n-2}, \qquad \forall \lambda \in (0,1).
\end{aligned}
\]
Letting $\lambda\to0$, we obtain
\[
    A(s)-\sigma_{n-1}
    \leq
    -\frac{N}{n-2}s^{n-2},
\]
which contradicts
\eqref{eq:A-small-s-expansion} since $N$ is arbitrary. Now,  we prove
\begin{align}\label{eq:zero limit}
    \lim_{s\to0^+}q(s)
    =
    -2(n-1)\sigma_{n-1}H_p.
\end{align}
By \eqref{eq:colding-third-monotonicity}, we know $q$ is increasing, thus  $q(0^+):=\lim_{s\to0^+}q(s)>-\infty$ exists.  Thus, for every fixed
$\varepsilon>0$, there exists $s_\varepsilon>0$ such that, whenever
$0<t<s<s_\varepsilon$,
\begin{equation*}
    |q(t)-q(0^+)|<\varepsilon, \quad \text{and }
    A(s)-A(\lambda s) = \int_{\lambda s}^{s}t^{n-3}q(t)\dd t, \qquad \forall \lambda \in (0,1).
\end{equation*}
Letting $\lambda\to0$, we obtain
\[
    \left|
        A(s)-\sigma_{n-1}
        -
        \frac{q(0^+)}{n-2}s^{n-2}
    \right|
    \leq
    \frac{\varepsilon}{n-2}s^{n-2},\qquad  \forall s\in (0, s_\varepsilon).
\]
Comparing this with \eqref{eq:A-small-s-expansion} gives 
\[
    \lim_{s\to0^+}q(s)=q(0^+)
    =
    -2(n-1)\sigma_{n-1}H_p.
\]
Substitute this into \eqref{eq:colding-third-monotonicity}, we get the \eqref{local representation}.  Again by \eqref{eq:colding-third-monotonicity} and $q\le 0$, we know 
$$L:=\lim_{R\to \infty}q(R) \quad  \text{ exits and} \qquad L\le 0.$$

Suppose that $L<0$. Then for any fixed $s_0>0$, 

\begin{align}\label{contra}
    A(R)-A(s_0)
    &=
    \int_{s_0}^R t^{n-3}q(t)\dd t
    \leq
    \frac{L}{n-2}
    \bigl(R^{n-2}-s_0^{n-2}\bigr)\to -\infty \qquad \text{ as }  \quad  R\to +\infty. 
\end{align}
But
$A'\leq0$ implies 
\begin{equation*}
    A(R)-A(s_0)\ge 0-
    \lim_{s\to0}A(s)
    =-\sigma_{n-1}, \qquad \forall R>s_0>0,
\end{equation*}
 which contradicts to \eqref{contra}. 
Therefore $L=0$. Letting $R\to \infty$ in \eqref{local representation} and noting $D\ge 0$, w get the conclusion. 
\end{proof}
\section{Local Quantitative Conical Function Estimate via $H_p$}\label{sec:epsilon_conical_function}

Throughout this section, we assume $(M^{n\ge 3},g)$ is a non-parabolic manifold and with  $\operatorname{Ric}\geq0$, $p\in M$ and  $G$ is the minimal normalized Green function satisfying \eqref{eq:intro_green_expansion}. Recall $b=G^{\frac{1}{2-n}}$.  For any $R>0$, we  define  $$\varepsilon_R=H_pR^{n-2},\qquad F=\frac{b^2}{2}, \qquad\text{ and } \quad  F_0=\frac{r^2}{2}.$$
The main result of this section is 
\begin{theorem}\label{thm:quantitative_conical_function}
Under the assumption of this section, there exists $C_n$ and $\varepsilon(n)>0$  such that  if $\varepsilon_{8R}\le \varepsilon(n)$, then 
\begin{align*}
   \left(\frac{1}{R^2}\|F-F_0\|_{L^\infty(B_p(R))}\right)^{\frac{1}{2}}+
    \frac{1}{R} \fint_{B_p(R)}
    |\nabla F-\nabla F_0|\dd V+\left( \fint_{B_p(R)}
    |\operatorname{Hess}F-g|^2\dd V\right)^{\frac{1}{2}}
    \leq
    C_n\varepsilon_{8R}^{\frac{1}{2(n+1)}}.
\end{align*}
\end{theorem}
This theorem is a collection  of the following Lemma \ref{lem:r-b-Linfty}, Lemma\ref{lem:gradient-F-comparison} and Lemma \ref{lem:mu-hessian-estimates}.

\begin{lemma}
\label{lem:L1-h-bound}  Under the assumption of this section,
for every $\rho>0$, $h=G-r^{2-n}$ is a bounded  Lipschitz function  satisfying
\begin{equation*}
    \int_{B_p(2\rho)} h \dd V
    \leq
    C_nH_p\rho^n.
\end{equation*}
\end{lemma}

\begin{proof}
By \eqref{eq:intro_green_expansion}, we know there exists $s_0>0$ such that if $r(x)\le s_0$, then 
$$|h(x)-H_p|+|\nabla h(x)|\le C.$$
Noting $\lim_{x\to \infty}G(x)=0$, we know $G$ is bounded on $ M\backslash B_p({s_0})$. Moreover,   Cheng-Yau gradient estimate implies $|\nabla G|\le \frac{C}{r(x)}|G|$, which is also bounded on $M\backslash B_p({s_0})$. As a result, we know 
$|\nabla h|\le |\nabla G|+(n-2)|r^{1-n}|\le C$ on $M\backslash B_p({s_0})$. So, $h$ is a bounded Lipschitz function.  By Proposition~2.1, $\Delta h = \Delta\bigl(G-r^{2-n}\bigr)\leq0$, and $h\geq0$ on $M$. Since $\Ric\geq0$ and $h$ is a nonnegative superharmonic
function, by \cite[Lemma~3.6]{Colding2012}, 
\begin{equation*}
    \frac{1}{\Vol(\partial B_p(s))} \int_{\partial B_p(s)} h \dd A \leq C(n) h(p) = C(n) H_p,\quad \forall s>0
\end{equation*}
which combining with Bishop-Gromov volume comparison theorem further implies 
\begin{equation*}
    \int_{B_p(\rho)} h \dd V \leq C(n) H_p \Vol(B_p(\rho))\le C H_p \rho^n.
\end{equation*}
\end{proof}

\begin{lemma}\label{lem:local-noncollapsing}
Under the assumption of this section, there exists $\varepsilon(n)>0$ such that, if $\varepsilon_L\leq\varepsilon(n)$, then
\begin{equation*}
    \Vol\bigl(B_p(L)\bigr)
    \geq
    \frac12\,\omega_nL^n.
\end{equation*}
Moreover, for every $x\in B_p(\frac{L}{4})$ and every $0<s\leq\frac{L}{4}$, one has
\begin{equation*}
    \Vol\bigl(B_x(s)\bigr)
    \geq
    v_ns^n,
    \qquad
    v_n=2^{-(n+1)}\omega_n.
\end{equation*}
\end{lemma}

\begin{proof}
Choose $\varepsilon(n)=\frac{n-1}{n}$. The conclusions follow from 
Theorem~\ref{thm:volume-ratio-estimates} and 
 Bishop--Gromov volume comparison theorem.
\end{proof}

\begin{lemma}\label{lem:b-r-comparison}Under the assumption of this section,
there exist $c_0(n)\in(0,1)$ and $\varepsilon(n)>0$ such that, if $\varepsilon_L<\varepsilon(n)$,
then for every $x\in B_p(\frac{L}{4})$,
\begin{equation*}
    c_0(n)r(x)\leq b(x)\leq r(x).
\end{equation*}
Consequently, for any $\rho \le\frac{ c_0(n)L}{4}$, there holds 
\begin{equation}
    \{b\leq\rho\}
    \subset
    B_p\left(\frac{L}{4}\right).
    \label{eq:sublevel-localization}
\end{equation} 
\end{lemma}
\begin{proof}Since $G\ge r^{2-n}$,
it suffices to prove the lower bound. Fix
\[
    x\in B_p\left(\frac{L}{4}\right),
    \qquad
    t:=r(x)>0.
\]
By Lemma~\ref{lem:local-noncollapsing},
$\Vol(B_x(t/8)) \geq v_nt^n$.
Since $B_x(t/8) \subset B_p(2t)$,
Lemma~\ref{lem:L1-h-bound} gives
\[
    \int_{B_x(t/8)}h\dd V
    \leq
    C_nH_pt^n.
\]
Hence there exists $y\in B_x(t/8)$ such that $h(y)\leq C_nH_p$.
Moreover,
$ y\in B_x\left(t/8\right),
    $ and $t=r(x) $ implies  $r(y)\geq 7t/8.$
Therefore
\[
    G(y)
    =
    r(y)^{2-n}+h(y)
    \leq
    C_nt^{2-n}+C_nH_p\le C_n(1+H_p t^{n-2})t^{2-n}\le  C_n t^{2-n},
\]
where  we use the assumption $H_pt^{n-2}\le \varepsilon_{L}\le 1$. 

Since $G$ is positive harmonic on $B_x(t/2)$, the Cheeger--Yau Harnack inequality
gives $G(x)\leq e^{C_nd(x,y)/t}\le  C_nG(y\leq C_nt^{2-n}$. Since $G=b^{2-n}$, we get
\[
    b(x)\geq c_0(n)t
    =
    c_0(n)r(x).
\]
Moreover, on
$\partial B_p(\frac{L}{4})$ one has $b\geq\frac{ c_0(n)L}{4}$.
Equivalently,
\[
    G
    \leq
    \left(
       \frac{ c_0(n)L}{4}
    \right)^{2-n}
    \qquad
    \text{on }\partial B_p\left(\frac{L}{4}\right).
\]
Since $\lim_{x\to \infty}G(x)=0$,
the maximum principle gives
\[
    G
    \leq
    \left(
       \frac{ c_0(n)L}{4}
    \right)^{2-n}
    \qquad
    \text{on }
    M\setminus B_p\left(\frac{L}{4}\right).
\]
Thus $b\geq\frac{ c_0(n)L}{4}$ on $M\setminus B_p(\frac{L}{4})$,
which proves \eqref{eq:sublevel-localization}.
\end{proof}

\begin{lemma}\label{lem:r-b-Linfty}Under the assumption of this section,
there exist $C_n<\infty$ and $\varepsilon(n)>0$ such that, if $\varepsilon_{8R}\leq\varepsilon(n)$,
then
\begin{equation*}
    \|F-F_0\|_{L^\infty(B_p(R))}
    \leq
    C_nR^2\varepsilon_{8R}^{\frac{1}{n+1}}.   
\end{equation*}
\end{lemma}

\begin{proof}
Fix $x\in B_p(R)$, since  $d:=r(x)-b(x)\geq0,$ we can assume $d>0$ without loss of generality. By $d\leq t\leq R$, $B_x( d/8 ) \subset B_p(2R)$. Thus for any $y\in B_x(d/8)$,
we have
\[
    |r(y)-r(x)|
    \leq
    d(x,y)
    \leq
    \frac{d}{8}\le \frac{r(x)}{8}.
\]
Hence
\begin{equation}\label{eq:ry_sim_t}
   \frac{7r(x)}{8}\le  r(y)\le \frac{9r(x)}{8},
    \qquad
    \forall y\in    B_x\left(\frac{d}{8}\right).
\end{equation}
Applying Lemma~\ref{lem:b-r-comparison} at scale $L=8R$,
we obtain
\begin{equation}\label{eq:by_sim_t}
 c_nr(x)\le  b(y)\le 2 r(x)
    \qquad, \forall y\in
    B_x\left(\frac{d}{8}\right).
\end{equation}
Thus by \eqref{eq:ry_sim_t} and \eqref{eq:by_sim_t}, $b(y)^{2-n}-r(y)^{2-n}\geq c_n r^{1-n}(x)\cdot(r(y)-b(y))$.
Since both $r$ and $b$ are $1$-Lipschitz\cite[Theorem 3.1]{Colding2012}, the function $r-b$ is
$2$-Lipschitz. Therefore, 
\[
\begin{aligned}
    (r-b)(y)
    &\geq
    (r-b)(x)-2d(x,y)
    \geq
    d-\frac{d}{4}
    =
    \frac{3d}{4},\qquad  \forall y\in B_x\left( \frac{d}{8}\right).
\end{aligned}
\]
Consequently,
\[
    h(y)
    =
    G(y)-r(y)^{2-n}
    = b(y)^{2-n}-r(y)^{2-n}
    \geq
    c_ndr^{1-n}(x),\qquad  \forall y\in B_x\left( \frac{d}{8}\right).
\]
Applying Lemma~\ref{lem:local-noncollapsing} at scale $8R$ gives $ \Vol(B_x(\frac{d}{8})) \geq c_nd^n$.
Moreover, Lemma~\ref{lem:L1-h-bound} gives
\[
     C_nd^{n+1}r^{1-n}(x)
    \leq
    \int_{B_x(d/8)}h\dd V\le \int_{B_p(2R)}h\dd V
    \leq
    C_nH_pR^n
    =
    C_n\varepsilon_{8R}R^2.
\]
Since $r(x)\leq R$, it follows that
\begin{align}\label{distance error}
    0\le r(x)-b(x)=d
    \leq
    C_nR\varepsilon_{8R}^{\frac{1}{n+1}},
\end{align}
which proves
\[
    |F-F_0|
    =
    \frac12|b-r|(b+r)
    \leq
    C_nR^2\varepsilon_{8R}^{\frac{1}{n+1}}.
\]
\end{proof}

\begin{lemma}\label{lem:gradient-F-comparison}
Under the assumption of this section, there exist $C_n<\infty$ and $\varepsilon(n)>0$ such that, if
$\varepsilon_{8R}\leq\varepsilon(n)$,
then
\begin{equation*}
    \fint_{B_p(R)}
    |\nabla F-\nabla F_0|\dd V
    \leq
    C_nR\varepsilon_{8R}^{\frac{1}{2(n+1)}}.
\end{equation*}
\end{lemma}

\begin{proof}
Observe that
\begin{equation}
\begin{aligned}
    |\nabla F-\nabla F_0|
    &=
    |b\nabla b-r\nabla r|
    \leq
    |b-r|\,|\nabla b|
    +
    r|\nabla b-\nabla r|.
\end{aligned}
\label{eq:gradient-F-pointwise}
\end{equation}
In geodesic polar coordinates centered at \(p\), for 
\(\theta\in \mathbb{S}^{n-1}\subset T_pM\), let
\[
\gamma_\theta(t)=\exp_p(t\theta),
\qquad
0<t<\min\{R,c(\theta)\},
\]
denote the radial minimizing geodesic, where \(c(\theta)\) is the cut
time in the direction \(\theta\), and define $f_\theta(t):=t-b(\gamma_\theta(t))$.
Since the cut locus has measure zero,
\begin{align*}
    \int_{B_p(R)}
    \bigl(1-\langle\nabla b,\nabla r\rangle\bigr)\dd V
    &=
    \int_{B_p(R)\setminus(\operatorname{Cut}(p)\cup\{p\})}
    \bigl(1-\langle\nabla b,\nabla r\rangle\bigr)\dd V
    \\
    &=
    \int_{\mathbb{S}^{n-1}}
    \int_0^{R\wedge c(\theta)}
    \bigl(
        1-\langle\nabla b,\nabla r\rangle
    \bigr)(\gamma_\theta(t))
    J(t,\theta)\dd t\dd \theta,
\end{align*}
where $J(t,\theta) \leq t^{n-1}$ by Bishop volume comparison theorem.
Before the cut time, $f_\theta'(t) = \bigl( 1-\langle\nabla b,\nabla r\rangle \bigr)(\gamma_\theta(t))$.
Therefore
\[
    \int_{B_p(R)}
    \bigl(1-\langle\nabla b,\nabla r\rangle\bigr)\dd V
    \leq
    R^{n-1}
    \int_{\mathbb{S}^{n-1}}
    f_\theta(R\wedge c(\theta))\dd\theta.
\]
By \eqref{distance error}, we know $f_\theta(R\wedge c(\theta))\le C_n R\varepsilon_{8R}^{\frac{1}{n+1}}$ and 
\[
    \int_{B_p(R)}
    \bigl(1-\langle\nabla b,\nabla r\rangle\bigr)\dd V
    \leq
    C_nR^n\varepsilon_{8R}^{\frac{1}{n+1}}.
\]
Moreover, since $|\nabla b|\leq1$,
\[
\begin{aligned}
    |\nabla b-\nabla r|^2
    &=
    1+|\nabla b|^2
    -
    2\langle\nabla b,\nabla r\rangle
    \leq
    2\bigl(
        1-\langle\nabla b,\nabla r\rangle
    \bigr).
\end{aligned}
\]
Using Lemma~\ref{lem:local-noncollapsing}, we obtain
\[
    \fint_{B_p(R)}
    |\nabla b-\nabla r|^2\dd V
    \leq
    C_n\varepsilon_{8R}^{\frac{1}{n+1}}.
\]
Applying the Cauchy--Schwarz inequality to
\eqref{eq:gradient-F-pointwise}, together with $r\leq R$ and \eqref{distance error}, yields
\[
    \fint_{B_p(R)}
    |\nabla F-\nabla F_0|\dd V
    \leq
    C_nR\varepsilon_{8R}^{\frac{1}{2(n+1)}}.
\]
\end{proof}

\begin{lemma}\label{lem:A-defect}Under the assumption of this section, 
for every $s>0$, there holds 
\begin{equation}
    0
    \leq
    \sigma_{n-1}-A(s)
    \leq
    C_nH_ps^{n-2}
    \label{eq:A-defect}
\end{equation}
and
\begin{equation}
    \int_{\{b\leq R\}}
    |\nabla b|^2
    \bigl(1-|\nabla b|^2\bigr)\dd V
    \leq
    C_nH_pR^{2n-2}.
    \label{eq:gradient-defect-sublevel}
\end{equation}
\end{lemma}

\begin{proof}
By \eqref{eq:zero limit}, we know
\[
    -2(n-1)\sigma_{n-1}H_p
    =
    \lim_{s\to0}q(s)
    \leq
    q(s)
    \leq0.
\]
Since
   $ A'(s)=s^{n-3}q(s)$ and $
    \lim_{s\to0}A(s)=\sigma_{n-1}$, 
integration gives \eqref{eq:A-defect}.
Furthermore,
\[
\begin{aligned}
    \sigma_{n-1}-A(s)
    &=
    I_1(s)-I_{|\nabla b|^2}(s)
    =
    s^{1-n}
    \int_{\{b=s\}}
    |\nabla b|
    \bigl(1-|\nabla b|^2\bigr)\dd A.
\end{aligned}
\]
By the coarea formula,
\begin{align*}
    \int_{\{b\leq R\}}
    |\nabla b|^2
    \bigl(1-|\nabla b|^2\bigr)\dd V
    &=
    \int_0^R
    s^{n-1}
    \bigl(\sigma_{n-1}-A(s)\bigr)\dd s.
\end{align*}
Using \eqref{eq:A-defect} immediately yields
\eqref{eq:gradient-defect-sublevel}.
\end{proof}

\begin{lemma}\label{lem:mu-hessian-estimates}
Under the assumption of this section,  there exist $C_n<\infty$ and $\varepsilon(n)>0$ such that, if $\varepsilon_{8R}\leq\varepsilon(n)$, then

\begin{equation*}
    \fint_{B_p(R)}
    |\operatorname{Hess}F-g|^2\dd V
    \leq
    C_n\varepsilon_{8R}^{\frac{1}{n+1}}.
\end{equation*}
\end{lemma}

\begin{proof}
Define $$T_b:= \operatorname{Hess}b^2-\frac{\Delta b^2}{n}g, \qquad  \mu:=|\nabla b|^2\qquad \text{ and } \quad 
\mu_B :=  \fint_{B_p(R)}\mu\dd V.$$
By $\Delta b^2 = 2n|\nabla b|^2 = 2n\mu$, we have
\[
    \operatorname{Hess}F-g
    =
    \frac12 T_b+(\mu-1)g.
\]
Noting$\langle T_b,g\rangle=0$, we know
\begin{equation}
\begin{aligned}
    |\operatorname{Hess}F-g|^2
    =
    \frac14|T_b|^2+n(1-\mu)^2.
\end{aligned}
\label{eq:hessian-F-pointwise}
\end{equation}
Since $B_p(R)\subset\{b\leq R\}$ and $\Vol(B_p(R)) \geq c_nR^n$, by \eqref{eq:Hp-D-identity},  we obtain
\begin{equation}
    \fint_{B_p(R)}
    |T_b|^2\dd V
    \leq R^{-n}
    \int_{\{b\leq R\}}
    |T_b|^2\dd V\le 
    C_n\varepsilon_R.
    \label{eq:Tb-average-bound}
\end{equation}
By the coarea formula and\eqref{eq:A-defect},
\begin{equation*}
    \int_{\{b\leq (1-C_n\varepsilon_{8R}^{\frac{1}{n+1}})R\}}
    \mu^2\dd V
    =
    \int_0^{(1-C_n\varepsilon_{8R}^{\frac{1}{n+1}})R}
    s^{n-1}A(s)\dd s
    \geq
    \left(\omega_n-C_n\varepsilon_R\right)\cdot\left((1-C_n\varepsilon_{8R}^{\frac{1}{n+1}})R\right)^n.
\end{equation*}
Noting \eqref{distance error}, we have $\{b\le (1-C_n\varepsilon_{8R}^{\frac{1}{n+1}})R)\}\subset B_p(R)$, which implies  
\begin{align*}
\frac{(1-C_n\varepsilon^{\frac{1}{n+1}}_{8R})\omega_nR^n}{\mathrm{Vol}(B_p(R))}&\le \fint_{B_p(R)}\mu^2=\fint_{B_p(R)}|\mu-\mu_B|^2+\mu_B^2.
\end{align*}
So, Proposition \ref{thm:volume-ratio-estimates} implies 
\begin{align*}
(1-\mu_B)^2\le 1-\mu_B^2&\le \fint_{B_p(R)}|\mu-\mu_B|^2+1-\frac{(1-C_n\varepsilon^{\frac{1}{n+1}}_{8R})\omega_nR^n}{\mathrm{Vol}(B_p(R))}\\
&\le \fint_{B_p(R)}|\mu-\mu_B|^2+C(n)\varepsilon_{8R}^{\frac{1}{n+1}}.
\end{align*}
Thus the Poincar\'e inequality (see, e.g., \cite[Corollary~7.1.11]{PetersenRiemannianGeometry}
and \cite[Theorem~2.11 and Remark~2.82]{CheegerColding1996}) gives
\begin{align*}
    \fint_{B_p(R)}
    (1-\mu)^2\dd V
    &=
    (1-\mu_B)^2
    +
    \fint_{B_p(R)}
    |\mu-\mu_B|^2\dd V\\
    &\le  2\fint_{B_p(R)}|\mu-\mu_B|^2+C(n)\varepsilon_{8R}^{\frac{1}{n
    +_1}}\\
    &\le C_nR^2\fint_{B_p(2R)}|\nabla\mu|^2\dd V+C(n)\varepsilon_{8R}^{\frac{1}{n+1}}.
\end{align*}
By calculation,  $b\nabla\mu = T_b(\nabla b,\cdot)$, thus $|\nabla\mu|^2 \leq b^{-2}|T_b|^2$.
Using $B_p(2R)\subset\{b\leq2R\}$
and \eqref{eq:Hp-D-identity} again, we obtain
\begin{align*}
    \int_{B_p(2R)}
    |\nabla\mu|^2\dd V
    \leq
    \int_{\{b\leq2R\}}
    b^{-2}|T_b|^2\dd V
    \leq
    C_nH_pR^{2n-4}.
\end{align*}
As a result, $\Vol(B_p(R)) \geq c_nR^n$ implies 
\begin{align}\label{eq:1-mu}
\fint_{B_p(R)}
    (1-\mu)^2\dd V 
    \le C(n)\varepsilon_{8R}^{\frac{1}{n+1}}.
\end{align}

Combining
\eqref{eq:hessian-F-pointwise},
\eqref{eq:Tb-average-bound} and \eqref{eq:1-mu}
gives
\[
    \fint_{B_p(R)}
    |\operatorname{Hess}F-g|^2\dd V
    \leq
    C_n\varepsilon_{8R}^{\frac{1}{n+1}}.
\]
\end{proof}

\section{Quantitative Metric Cone via Conical Function}\label{sec:Quantitative Metric Cone via conical function} In Cheeger-Colding theory, during the proof of "almost volume cone implies almost metric cone",  almost conical function is introduced to give the almost cosine law, which further implies the almost metric cone property.\cite{Colding1997}. In this section, we follow \cite{Colding1997} \cite{CheegerColding1996}\cite{XuZhou23}\cite{Jiang} to write a version with precise quantitative order. The polynomial power is necessary for application in the next section. For $F\in C^2(B_p(4))$, we define 
\[
  F_0:=\frac{r^2}{2},  \qquad   e_1:=|\nabla F-\nabla F_0|,
    \qquad
    e_2:=|\hess F-g|.\]
The main result is  
\begin{theorem}\label{thm:three} Let $M$ be a complete non-compact Riemannian manifold and  $\mathrm{Ric\ge 0}$ on $B_p(4)\subset M$. There exists   $\delta_n>0,C_n<\infty$ such that if there exists $F$ satisfying 
\begin{equation}\label{eq:threeinput}
\delta:=\fint_{B_p(3/2)}(e_1+e_2)\dd V\leq\delta_n,
\end{equation}
then there exists a compact metric space $(Z,d_Z)$ of diameter at most $\pi$ such that
\begin{equation}\label{eq:threeconclusion}
d_{GH}\left(B_p\left(\frac{1}{10}\right),B^{C(Z)}_o\left(\frac{1}{10}\right)\right)
\leq C_n\delta^{\beta_n},\qquad
\beta_n=\frac{1}{44(4n+1)},
\end{equation}
where $C(Z)$ is the cone with link $Z$ and $d_{GH}$ is the pointed Gromov-Hausdorff distance. 
\end{theorem}

    No non-collapsing condition is assumed.

\begin{lemma}\label{lem:taylor_ode}
    Suppose $\gamma:[0,L] \to B_p(3)\setminus B_{p}(\sigma/2)\subset M$ is a unit-speed minimizing geodesic and $F$ is a smooth function on $B_p(3)\setminus B_{p}(\sigma/2)$. Then 
    \begin{equation}\label{eq:talor_ode}
        F(\gamma(L)) - F(\gamma(0)) = L \langle\grad F (\gamma(0)), \gamma'(0) \rangle +\frac{L^2}{2} + R,
    \end{equation}
    where 
    \begin{equation}\label{eq:taylor_ode_R}
        |R| \leq L \int_{\gamma} |\hess F - g|\dd s 
    \end{equation}
\end{lemma}
\begin{proof}
    Set $f(t):= F \circ \gamma(t)$. Then 
    \begin{equation*}
        \begin{aligned}
            &f(L) - f(0) = L\langle \grad F(\gamma(0)),\gamma'(0) \rangle + \frac{L^2}{2} + \int_0^L \int_0^s (\hess F -g) (\gamma'(t), \gamma'(t)) \dd t \dd s\\
        &= L\langle \grad F(\gamma(0)),\gamma'(0) \rangle + \frac{L^2}{2} + \int_0^L(L-t) (\hess F -g) (\gamma'(t), \gamma'(t)) \dd t, 
        \end{aligned}
    \end{equation*}
    and the conclusion follows from
    \begin{equation*}
        \big|\int_0^L (L-t) (\hess F -g) (\gamma'(t), \gamma'(t)) \dd t\big| \leq L \int_{\gamma} |\hess F - g|\dd s.
    \end{equation*}
\end{proof}
For any 
$x,y,z\in B_p(1)$, define the radical excess of $x,z$ by 
$$e(x,z)=d(x,z)-|r(z)-r(x)|.$$
Define the comparison angle(in fact, cosine of the angle) between $x,y$ by 
$$K(x,y):= \frac{r(x)^2 +r(y)^2 - d(x,y)^2}{2r(x)r(y)}.$$
The following proposition estimates the difference of the comparison angles calculating by  pair of points in different scales when the radical excess is small.  
\begin{proposition}\label{prop:four_points_estimate}Let $M$ be a complete non-compact Riemannian manifold and  $\mathrm{Ric\ge 0}$ on $B_p(4)\subset M$. There exist $\sigma_0(n)>0$ and $C_n<\infty$ such that  for any $\sigma\in (0,\sigma_0(n))$ if  there exists   $F\in C^2(B_p(3))$ satisfying 
\begin{equation*}
    \delta = \fint_{B_{p}(3/2)} \left( e_1+e_2 \right) \dd V, \qquad \delta\le\sigma^{44(4n+1)}.
\end{equation*}
Then for any  $x_i, z_i \in B_{p}(\frac{1}{2})\backslash B_p(\sigma)$ with $e(x_i,z_i)\leq C_0 \sigma^{43}, i =0, 1$, 
there holds 
\begin{equation*}
    |K(x_0,x_1)-K(z_0,z_1)|\leq C(n,C_0) \sigma^{18}.
\end{equation*}
\end{proposition}
\begin{proof}
 Noting 
\begin{align}\label{eq:four_K_two_terms}
|K(x_0,x_1)-K(z_0,z_1)|\le |K(x_0,x_1)-K(z_0,x_1)|+|K(z_0,x_1)-K(z_0,z_1)|.
\end{align}
Below we  estimate $|K(x_0,x_1)-K(z_0,x_1)|$ and the estimate of the second term is similar. 
Now, let $\rho = \sigma^{44}$ and consider the following product space
$$P := B_{x_0}(\rho) \times B_{x_1}(\rho) \times B_{z_0}(\rho).$$ 
Let $\mu^3$ denote the product measure on $P$, and  $\Gamma_{xy}$ denote a minimizing geodesic joining $x$ and $y$. We first construct a full-measure subset $W_0$ of $P$ on which the required first variation formulas can be applied, and $\grad r$ exists a.e. on the geodesics connecting the perturbed points chosen from $W_0$.

By the segment inequality\cite[Theorem 2.11]{CheegerColding1996}
\begin{equation*}
\begin{aligned}
    \int_{B_{x_0}(\rho)} \int_{B_{z_0}(\rho)} \int_{\Gamma_{a_0b_0}} 1_{\operatorname{Cut}(p)} \dd s \dd \mu(b_0) \dd \mu(a_0)  &= 0,\\
        \int_{B_{x_1}(\rho)}\int_{B_{x_0}(\rho)} \int_{B_{z_0}(\rho)} \Big(\int_{\Gamma_{a_0b_0}} 1_{\operatorname{Cut}(a_1)} \dd s \Big) \dd \mu(b_0) \dd \mu(a_0)\dd \mu(a_1)&=0.   
    \end{aligned}
\end{equation*}
Thus
\begin{equation}\label{eq:four_cut_penalty}
    \begin{aligned}
        \mu^3\Big(\left\{(a_0,a_1,b_0)\in P | \ \mathcal{H}^{1}\big(\Gamma_{a_0b_0}([0,1]) \cap \operatorname{Cut}(p)\big)> 0\right\}\Big)= 0, \\
        \mu^3\Big(\left\{ (a_0,a_1,b_0)\in P | \ \mathcal{H}^{1}\big(\Gamma_{a_0b_0}([0,1]) \cap \operatorname{Cut}(a_1)\big)> 0\right\}\Big)= 0. 
    \end{aligned}
\end{equation}
By applying the segment inequality twice,
\begin{equation}\label{eq:four_cut_estimate_3}
    \begin{aligned}
        \int_{B_{x_0}(\rho)} \int_{B_{z_0}(\rho)} \int_{\Gamma_{a_0 b_0}} \int_{B_{x_1}(\rho)} \left(\int_{\Gamma_{w a_1}} 1_{\operatorname{Cut}(p)} \dd s \right) \dd \mu(a_1) \dd \mathcal{H}^1(w) \dd \mu(b_0) \dd \mu(a_0) = 0.
    \end{aligned}
\end{equation}
So, $\operatorname{Cut}(p)$ intersections have zero
$\mathcal H^1$-measure for almost every
choice of the relevant endpoints
and almost every intermediate point
along the selected geodesics.


Let $W_0 \subset P$ be the set of all $(a_0,a_1,b_0)$ satisfying the following conditions:
\begin{enumerate}
    \item $a_0$ and $b_0$ are jointed by unique minimizing geodesic;
    \item $\mathcal{H}^1(\Gamma_{a_0b_0}([0,1]) \cap \operatorname{Cut}(p)) = 0$;
    \item $\mathcal{H}^1(\Gamma_{a_0b_0}([0,1]) \cap \operatorname{Cut}(a_1)) = 0$;
    \item For $\mathcal{H}^1$ a.e. $w \in \Gamma_{a_0b_0}([0,1])$, $\mathcal{H}^{1}(\Gamma_{wa_1}([0,1]) \cap \operatorname{Cut}(p)) = 0$.
\end{enumerate}
By \eqref{eq:four_cut_penalty},
\eqref{eq:four_cut_estimate_3},
and Fubini's theorem, we conclude that $\mu^3(P\setminus W_0)=0$.

Let $q = |\hess F - g| + |\grad F - \grad F_0|$.
For $(a_0,a_1,b_0) \in W_0$, define
\[\begin{aligned}
    A(a_0,b_0)
    &= \int_{\Gamma_{a_0 b_0}} q \dd s,
    \\
    B_0(a_0,b_0,a_1)
    &= \int_{\Gamma_{a_0 b_0}}
    \left(
        \int_{\Gamma_{w a_1}} q\dd s
    \right)
    \dd\mathcal{H}^1(w).
\end{aligned}\]
By the segment inequality \label{lem:local_relative_volume_comparison}, 
\begin{equation}\label{eq:four_A_i_estimate}
    \begin{aligned}
        \fint_{B_{x_0}(\rho)} &\fint_{B_{z_0}(\rho)} A(a_0,b_0) \dd \mu(b_0) \dd \mu(a_0)\\ &\leq C_n \frac{\Vol(B_{x_0}(\rho)) + \Vol(B_{z_0}(\rho))}{\Vol(B_{x_0}(\rho))\Vol(B_{z_0}(\rho))} \Vol(B_{p}(3/4)) \fint_{B_{p}(3/4)} q \dd \mu 
        \leq C_n \rho^{-n} \delta, 
    \end{aligned}
\end{equation}

\begin{equation}\label{eq:four_B_0_estimate}
\begin{aligned}
&\fint_{B_{x_1}(\rho)}
\fint_{B_{x_0}(\rho)}
\fint_{B_{z_0}(\rho)}
B_0(a_0,b_0,a_1)
\,\dd\mu(b_0)\dd\mu(a_0)\dd\mu(a_1)\\
&\quad\leq
C_n\rho^{-n}
\fint_{B_{x_1}(\rho)}
\fint_{B_p(3/4)}
\int_{\Gamma_{wa_1}}q\,\dd s
\,\dd\mu(w)\dd\mu(a_1)\\
&\quad\leq
C_n\rho^{-n}\left(\frac{\Vol(B_p(3/2))}{\Vol(B_{x_1}(\rho))} + \frac{\Vol(B_p(3/2))}{\Vol(B_{p}(3/4))}\right)
\fint_{B_p(3/2)}q\,\dd\mu=C_n\rho^{-2n}\delta.
\end{aligned}
\end{equation}
Let 
$  J: P \to \R,$ $J = A_0 +A_1 +B_0.$
Combining \eqref{eq:four_A_i_estimate}, \eqref{eq:four_B_0_estimate} and the Chebyshev inequality, we obtain
\begin{equation}\label{eq:four_chebyshev_estimate_of_P}
    \frac{\mu^3(\{\alpha\in P|J(\alpha) > \sigma^{44}\})}{\mu^3(P)} \leq C_n \sigma^{-44}\delta(\rho^{-2n}+ \rho^{-n})\le \frac{1}{2}
\end{equation}
Consequently, there exists a quadruple of perturbed points
\begin{equation}\label{eq:four_good_points_selection}
(\tilde{x}_0,\tilde{x}_1,\tilde{z}_0) \in \{\alpha\in W_0:J(\alpha)\le \sigma^{44}\}.  
\end{equation}
By definition of $K$ and the triangle inequality,  
\begin{equation}\label{eq:four_perturbed_K_estimate}
\begin{aligned}
    \big|K(x_0,x_1)-K(z_0,x_1)\big|&\leq \big( \big| K(x_0,x_1)-K(\tilde{x}_0,\tilde{x}_1)\big| +\big|K(\tilde{x}_0,\tilde{x}_1) - K(\tilde{z}_0,\tilde{x}_1)\big|\\ &\qquad \qquad \qquad \qquad \qquad \quad \ \ \ +
    \big| K(z_0,x_1)-K(\tilde{z}_0,\tilde{x}_1)\big|
    \big)\\
    &\leq C_n \sigma^{40} + \big| K(\tilde{x}_0,\tilde{x}_1)- K(\tilde{z}_0,\tilde{x}_1) \big|. 
\end{aligned} 
\end{equation}
Let $L:=d(\tilde{x}_0,\tilde{z}_0)$, and $\Gamma:[0,L]\to B_{p}(3/2)$ denote the minimizing geodesic connecting $\tilde{x}_0$ and $\tilde{z}_0$. By the radial excess assumption and the triangle
inequality,
we have
$r\left( \Gamma(t) \right)\geq \frac{\sigma}{2}$ for $t\in[0,L]$. Denote
\begin{equation*}
    K(t) := K(\Gamma(t), \tilde{x}_1),  \quad S(t) := d(\Gamma(t), \tilde{x}_1)^2, \quad r(t) := r(\Gamma(t)), \quad b:=r(\tilde{x}_1). 
\end{equation*}
Then
\begin{equation}\label{eq:four_S_K_relation}
    S(t) = b^2 +r(t)^2 -2br(t)K(t) \quad \text{ and } \quad \big| K(\tilde{x}_0,\tilde{x}_1)- K(\tilde{z}_0,\tilde{x}_1) \big| \leq \int_0^L |K'(t)|\dd t 
\end{equation}
Differentiating with respect to $t$ gives 
\begin{equation}\label{eq:four_S'_K'_relation}
    \begin{aligned}
        S'(t) = &2 r(t) r'(t)
        -2b\big(r'(t) K(t)+ r(t)K'(t)\big).
    \end{aligned}
\end{equation}
$\mathcal{H}^1(\Gamma_{\tilde{x}_0\tilde{z}_0}([0,1]) \cap \operatorname{Cut}(\tilde{x}_1)) = 0$ implies that, for $\mathcal{H}^1$ a.e. $w \in \Gamma_{\tilde{x}_0\tilde{z}_0}([0,1])$, the minimizing geodesic $\Gamma_{w,\tilde{x}_1}$ is unique, thus  
\begin{equation*}
    \frac{d}{\dd t} \left. d(\Gamma(t),\tilde{x}_1)\right|_t =  -\langle\Gamma'(t), T \rangle, \quad \text{for }\mathcal{H}^1 \text{ a.e. } t \in[0,L], 
\end{equation*}
where $T(t)=\Gamma_{w,\tilde{x}_1}'(0)$ for $w=\Gamma(t)$.  Calculation gives
\begin{equation}\label{eq:four_S'_first_variation}
    \begin{aligned}
        S'(t) 
        = -2 d(\Gamma(t),\tilde{x}_1) \langle \Gamma'(t), T\rangle = -2 d(\Gamma(t),\tilde{x}_1) \big(\langle \Gamma'(t) - \frac{\grad F}{r}(\Gamma(t)), T\rangle + \langle\frac{\grad F}{r}(\Gamma(t)),T \rangle \big)
    \end{aligned}
\end{equation}
for a.e. $t \in [0,L]$. 
Applying Lemma~\ref{lem:taylor_ode} with $\gamma = \Gamma_{\Gamma(t), \tilde{x}_1}$, 
\begin{equation}\label{eq:four_taylor_ode}
    d(\Gamma(t), \tilde{x}_1)\langle \grad F(\Gamma(t)), T\rangle = F(\tilde{x}_1) - F(\Gamma(t)) -\frac{S(t)}{2} - R(t).
\end{equation}
Substitute \eqref{eq:four_taylor_ode} into \eqref{eq:four_S'_first_variation} yields, 
\begin{equation*}
\begin{aligned}
    \frac{S'(t)}{2} = - \frac{1}{r(t)}\big(F(\tilde{x}_1) - F(\Gamma(t)) -\frac{S(t)}{2} - R(t) \big) - d(\Gamma(t),\tilde{x}_1) \langle \Gamma'(t) - \frac{\grad F}{r}(\Gamma(t)), T\rangle.
\end{aligned}
\end{equation*}
Since our perturbed point is chosen from $W_0$, for a.e. $w \in \Gamma([0,L])$, $\grad F_0$ exists $\mathcal{H}^1$ a.e. on $\Gamma_{w\tilde{x}_1}([0,1])$.
Let $\xi(x):= F(x)-F_0(x)$, then $\grad \xi$ exists $\mathcal{H}^1$ a.e. on $\Gamma_{w\tilde{x}_1}([0,1])$ for $\mathcal{H}^1$ a.e. $w \in \Gamma([0,L])$.
Then for a.e. $w\in \Gamma([0,L])$, 
\begin{equation}\label{eq:four_using_e_1_control_e_0}
    \begin{aligned}
        |\xi(\tilde{x}_1) - \xi(w)| \leq \int_{\Gamma_{w\tilde{x}_1}}\Big|\grad F - \grad F_0 \Big| \dd s
    \end{aligned}
\end{equation}
Thus for a.e. $t\in [0,L]$, there holds
\begin{equation*}
\begin{aligned}
    \frac{S'(t)}{2} =- \frac{1}{r(t)}\big( \frac{b^2 - r(t)^2}{2} + \xi(\tilde{x}_1)-\xi(\Gamma(t)) -\frac{S(t)}{2} - R(t) \big) - d(\Gamma(t),\tilde{x}_1) \langle \Gamma'(t) - \frac{\grad F}{r}(\Gamma(t)), T\rangle.
\end{aligned}
\end{equation*}
Let
\begin{equation}\label{eq:four_def_E_s}
    E_s := -2\frac{\xi(\tilde{x}_1) - \xi(\Gamma(t))}{r(t)}+2\frac{R(t)}{r(t)} - 2 d(\Gamma(t),\tilde{x}_1) \langle \Gamma'(t) - \frac{\grad F}{r}(\Gamma(t)), T\rangle,
\end{equation}
then by \eqref{eq:four_S_K_relation},
\begin{equation}\label{eq:four_S'_E_s}
\begin{aligned}
    S'(t) = - \frac{1}{r(t)}\big( b^2 - r(t)^2 -S(t)\big) + E_s= \frac{1}{r(t)}\big(2r(t)^2 -2br(t)K(t)\big) + E_s.
\end{aligned}
\end{equation}
Combining \eqref{eq:four_S'_E_s} with \eqref{eq:four_S'_K'_relation}, we have 
\begin{equation*}
\begin{aligned}
    &\frac{1}{r(t)}\big(2r(t)^2 -2br(t)K(t)\big) + E_s = 
    2 r(t)  r'(t) -2b\big(r'(t) K(t)+ r(t)K'(t)\big),
\end{aligned}
\end{equation*}
from which one can solve
\begin{equation*}
    \begin{aligned}
        K'(t) = \big(\frac{K(t)}{r(t)} -\frac{1}{b}\big)\big(1 - r'(t)\big) -\frac{E_s}{2b r(t)}.
    \end{aligned}
\end{equation*}

Now we estimate $\int_0^L |K'(t)|\dd t$ term by term. For the first term and second term of \eqref{eq:four_def_E_s}, since 
$r(t)\geq \sigma/2$, and by \eqref{eq:four_using_e_1_control_e_0} and  \eqref{eq:four_good_points_selection}, 
\begin{equation}\label{eq:four_E_s_1_2_estimate}
    \begin{aligned}
    \int_0^L \frac{|\xi(\tilde{x}_1) - \xi(\Gamma(t))|}{r(t)} \dd t &\leq C_n\sigma^{-1}\int_{\Gamma_{\tilde x_0\tilde z_0}}
\int_{\Gamma_{w\tilde x_1}}
|\nabla F-\nabla F_0|\,ds\,d\mathcal H^1(w)\leq C_n \sigma^{-1} \Lambda,\\
        \int_0^L \frac{|R(t)|}{r(t)} \dd t &\leq C_n\sigma^{-1} \int_0^L |R(t)|\dd t \leq C_n \sigma^{-1}\Lambda. 
        \end{aligned}
\end{equation}
For the third term of \eqref{eq:four_def_E_s}, since $d(\Gamma(t), \tilde{x}_1)\leq C$,
\begin{align}
        \int_0^L\Big| d(\Gamma(t),\tilde{x}_1) \langle \Gamma'(t) - \frac{\grad F}{r}(\Gamma(t)), T\rangle\Big|\dd t &\leq C_n \int_0^L |\Gamma'(t) - \frac{\grad F}{r}(\Gamma(t))|\dd t \notag\\
        &\leq C_n \int_0^L \big|\Gamma'(t) - \grad r (\Gamma(t))\big| + \big|\frac{\grad F_0-\grad F}{r}(\Gamma(t))\big|\dd t\notag\\
        &\leq C_n\left( \int_0^L |\Gamma'(t) - \grad r (\Gamma(t))| \dd t + \sigma^{-1} \Lambda\right)\notag\\
    &\leq C_n\left( L^{\frac{1}{2}} \left(\int_0^L|\Gamma'(t) - \grad r (\Gamma(t))|^2 \dd t\right)^{\frac{1}{2}} \sigma^{-1}\Lambda\right)\notag\\
    &\leq C_n  \left(\int_0^L2\Big(1 - \big(r(t)\big)'\Big)\dd t\right)^{\frac{1}{2}}+C_n\sigma^{-1}\Lambda\notag\\ 
    &= C_n  \left(L-[r(L) -r(0)]\right)^{\frac{1}{2}}+C_n\sigma^{-1}\Lambda\notag\\
    &=C_n e(\tilde{x}_0,\tilde{z}_0)^{\frac{1}{2}}+C_n\sigma^{-1}\Lambda \leq C_n \left(\sigma^{\frac{43}{2}}+\sigma^{-1}\Lambda \right),\label{eq:four_E_s_third_estimate}
\end{align}
Combining \eqref{eq:four_E_s_1_2_estimate} and \eqref{eq:four_E_s_third_estimate}, we obtain 
\begin{equation}\label{eq:four_E_s_estimate}
    \begin{aligned}
        \int_0^L\frac{|E_s|}{b r(t)}\dd t \leq C_n \sigma^{-2}\left(\sigma^{-1}\Lambda +\Lambda +\sigma^{\frac{43}{2}}\right) \leq C_n (\sigma^{-3}\Lambda +\sigma^{\frac{39}{2}}). 
    \end{aligned}
\end{equation}
By the triangle inequality, $|K(t)|\leq 1$. Thus 
\begin{equation}\label{eq:four_int_K'_estimate_prepare}
\begin{aligned}
    \Bigg|\int_0^L\big( \frac{K(t)}{r(t)} -\frac{1}{b} \big) \big(1 - r'(t)\big) \dd t\Bigg|
    \leq  C_n \sigma^{-1}\int_0^L 1 - r'(t) \dd t
    \leq C_n\sigma^{-1}e(\tilde{x}_0, \tilde{z}_0) \leq C_n \sigma^{42}. 
\end{aligned}
\end{equation}
Finally, by \eqref{eq:four_E_s_estimate} and \eqref{eq:four_int_K'_estimate_prepare},
    \begin{align*}
        \big|K(\tilde{x}_0,\tilde{x}_1) - K(\tilde{z}_0,\tilde{x}_1)\big| &\leq \int_0^L|K'(t)|\dd t \leq C_n(\sigma^{42}+\sigma^{\frac{39}{2}}+\sigma^{-3}\Lambda)\leq C_n \sigma^{\frac{39}{2}}.
    \end{align*}
Combining  with \eqref{eq:four_perturbed_K_estimate} and \eqref{eq:four_K_two_terms}, we obtain the conclusion.
\end{proof}

\begin{lemma}
\label{lem:radial_extension}Let $M$ be a complete non-compact Riemannian manifold and  $\mathrm{Ric\ge 0}$ on $B_p(4)\subset M$. There exist constants $\sigma_0(n)>0$ and $C_n<\infty$
such  that for any  $\sigma\in (0,\sigma_0(n))$, if there exits $F\in C^2(B_p(3))$ satisfying  
\begin{equation*}
\delta:=\fint_{B_p(3/2)}(e_1+e_2)\dd V
\le \sigma^{44(4n+1)}.
\end{equation*}
Then for every $z\in S_\sigma:=\{r=\sigma\}$ and
$ T\in [0,\frac15]$, there exists $y\in B_p(1/2)$
such that
\begin{equation*}
|r(y)-(\sigma+T)|+|d(z,y)-T|+e(z,y)
\le C_n \sigma^{44}.
\end{equation*}
\end{lemma}

\begin{proof} Set $\rho=\sigma^{44} ,V_0:=\Vol(B_p(3/2))$ and fix
 $z\in S_\sigma$ and $T\in[0,1/5]$.
 Define
\[
G:=
\left\{
w\in B_\rho(z)\setminus
(\operatorname{Cut}(p)\cup\{p\}):
e_1(w)\le\rho
\right\}.
\]
 Chebyshev's inequality gives
\begin{align*}
\frac{\Vol(B_\rho(z)\setminus G)}
     {\Vol(B_\rho(z))}
&\le
\frac{1}{\rho\Vol(B_\rho(z))}
\int_{B_\rho(z)}e_1\dd V \le C_n\delta\rho^{-n-1} \le C_n\rho^{3n}.
\end{align*}
Taking $\sigma_0(n)$ sufficiently small, we obtain
$
\Vol(G)\ge\frac12\Vol(B_\rho(z))$.
For each $w\in G$, consider
\[
\mathcal C_w:=
\left\{
v\in S_wM:
\angle(v,\nabla r(w))\le\rho
\right\}.
\]
Its spherical measure satisfies $\sigma_w(\mathcal C_w)\ge c_n\rho^{n-1}$. Let $\lambda$ denote the Liouville measure
on the unit tangent bundle $SM$, normalized by $\dd\lambda(w,v)=\dd\sigma_w(v)\dd V(w)$.
Define  $$\mathcal A:=
\left\{
(w,v)\in SM:
w\in G,\ v\in\mathcal C_w
\right\}.$$
Then volume comparison theorem imlies
\begin{equation}
\begin{aligned}
\lambda(\mathcal A)
&=\int_G\sigma_w(\mathcal C_w)\dd V(w)
\ge c_n\rho^{n-1}\Vol(G)
\ge c_n\rho^{2n-1}V_0.
\label{eq:extension-Liouville-size}
\end{aligned}
\end{equation}
For $(w,v)\in\mathcal A$, let
$\gamma_{w,v}:[0,T]\to M$ be the unit-speed
geodesic satisfying
\[
\gamma_{w,v}(0)=w,
\qquad
\dot\gamma_{w,v}(0)=v.
\]
Since $r(w)\le \sigma+\rho$, we have, for every $t\le T\le \frac 15$,
\[
r(\gamma_{w,v}(t))
\le r(w)+t
\le \sigma+\rho+\frac15<\frac12.
\]
Let $\phi_t:SM\to SM $ be the geodesic flow. Then, $\pi(\phi_T(\mathcal A))\subset B_p(1/2)$, and the  Liouville measure is invariant under $\phi_t$, wehre $\pi:SM\to M$ is the bundle projection. So, 
\begin{equation}
\begin{aligned}
\int_{\mathcal A}
e_1(\gamma_{w,v}(T))\dd \lambda
&=
\int_{\phi_T(\mathcal A)}
e_1(\pi(\xi))\dd \lambda(\xi)
\le
|\mathbb S^{n-1}|
\int_{B_p(1/2)}e_1\dd V
\le C_n\delta V_0.
\label{eq:extension-terminal-e1}
\end{aligned}
\end{equation}
Similarly, by Fubini-Tonelli's theorem and
Liouville invariance,
\begin{align}
\int_{\mathcal A}
\int_0^T e_2(\gamma_{w,v}(t))\dd t\dd
\lambda(w,v)
=
\int_0^T
\int_{\mathcal A}
e_2(\gamma_{w,v}(t))\dd \lambda\dd t
\le C_n\delta V_0.
\label{eq:extension-flow-e2}
\end{align}
 Combining
\eqref{eq:extension-Liouville-size},
\eqref{eq:extension-terminal-e1},
and \eqref{eq:extension-flow-e2},
Chebyshev's inequality shows that the relative
measure of the set where either
\[
e_1(\gamma_{w,v}(T))>\rho
\quad
\text{or} \quad
\int_0^T e_2(\gamma_{w,v}(t))\dd t>\rho
\]
is at most
$
C_n\delta\rho^{-2n}\le C_n\rho^{2n+1}<\frac12.
$ Moreover, since $\operatorname{Cut}(p)\cup\{p\}$ has zero volume and $\phi_T$ preserves $\lambda$,
the set of initial conditions for which
$\gamma_{w,v}(T)\in\operatorname{Cut}(p)\cup\{p\}$
has zero Liouville measure.
So, there exists
$(w,v)\in\mathcal A$ such that,$y:=\gamma(T)\notin\operatorname{Cut}(p)\cup\{p\}$ and $\gamma:=\gamma_{w,v}$ satisfy
\begin{equation}\label{eq:extension-good-geodesic}
\begin{aligned}
e_1(w)\le\rho, \quad 
e_1(y)\le\rho, \quad
|v-\nabla r(w)|&\le\rho, \quad
d(w,z)\le\rho,\quad 
\int_0^T e_2(\gamma(t))\dd t&\le\rho.
\end{aligned}
\end{equation}
Denoting 
$
V(t):=
\nabla F(\gamma(t))
-(r(w)+t)\dot\gamma(t),
$ Then, 
\begin{equation*}
\left|V'(t)\right|=
\left|(\hess F)^\sharp(\dot\gamma(t))
-\dot\gamma(t)\right|
=
\left|(\hess F-g)^\sharp(\dot\gamma(t))\right|\le e_2(\gamma(t)).
\end{equation*}
Integrating this, we obtain
\begin{align*}\label{eq:extension-gradient-transport}
\left|
\nabla F(y)-(r(w)+T)\dot\gamma(T)
\right|
&\le |V(0)|+
\int_0^T e_2(\gamma(t))\dd t\le |\nabla F(w)-r(w)v|+\rho\\
&\le e_1(w)+r(w)\rho+\rho\le 3\rho.
\end{align*}
Since $y\notin\operatorname{Cut}(p)\cup\{p\}$, we know $|\nabla r(y)|=1$ and
\begin{align*}
\left|
r(y)\nabla r(y)
-(r(w)+T)\dot\gamma(T)
\right|
\le
e_1(y) +
|\nabla F(y)-(r(w)+T)\dot\gamma(T)|\le4\rho.
\end{align*}
Noting $|\dot\gamma(T)|=|\nabla r(y)|=1$, we get
\begin{equation}
|r(y)-(r(w)+T)|\le4\rho.
\end{equation}
Combining this with $ d(w,z)\le\rho$, the triangle inequality implies 
\begin{align*}
|d(y,z)-T|\le |d(y,w)-T|+d(z,w)\le T-d(y,w)+\rho\le T+r(w)-r(y)+\rho\le 5\rho.
\end{align*}
and 
\begin{align*}
|r(y)-(\sigma+T)|\le d(z,w)+|r(y)-(r(w)+T)|\le 5\rho.
\end{align*}
So,  the radical excess satisfies
\begin{equation*}
0\le e(z,y)
=d(z,y)-[r(y)-\sigma] \le (T+\rho)-(T-5\rho) =6\rho.
\end{equation*}
Collecting the last three estimates, we obtain the conclusion.
\end{proof}
\begin{proof}[Proof of Theorem \ref{thm:three}] 
Set $\delta_n=(\sigma_0(n))^{44(4n+1)}$ and $\sigma=\delta^{\frac{1}{44(4n+1)}}$. By Proposition \ref{prop:four_points_estimate} and Lemma \ref{lem:radial_extension}, we know for any  $x_i, z_i \in B_{p}(\frac{1}{2})\backslash B_p(\sigma)$ with $e(x_i,z_i)\leq C_0 \sigma^{43}, i =0, 1$, 
there holds 
\begin{equation}\label{eq:four_K_estimate}
    |K(x_0,x_1)-K(z_0,z_1)|\leq C_n \sigma^{18}.
\end{equation}
and for every $z\in Z:= S_\sigma=\{r=\sigma\}$ and
$ T\in [0,\frac15]$, there exists $y\in B_p(1/2)$
such that
\begin{equation}\label{eq:extension-radial-defect}
|r(y)-(\sigma+T)|+e(z,y)
\le C_n \sigma^{44}.
\end{equation}
\textbf{Step 1.}
For $z,w\in Z$, define
\begin{equation}\label{eq:chain-metric}
\widehat\ell_\sigma(z,w):=\frac1\sigma\inf\left\{\sum_{i=1}^N d(z_{i-1},z_i):
z_0=z, z_N=w, d(z_{i-1},z_i)\le \sigma^{3}, 1\le i\le N\right\},
\end{equation}
and  $d_Z(z,w):=\min\{\pi,\widehat\ell_\sigma(z,w)\}$. Then, $d_Z$ is  a metric such that $\mathrm{diam}_{d_Z}(Z)\le\pi$. Noting that   $d(z,w)\le \sigma^{3}$ implies 
$
 d_Z(z,w)=\frac{d(z,w)}\sigma,
$
we know  $d$ and $d_Z$ gives the same topology, which implies $(Z,d_Z)$ is compact. 
Construct the metric cone over $Z$ as following:
\begin{equation*}
 C(Z)=([0,\infty)\times Z)/(\{0\}\times Z),\qquad
 d_C((s,z),(t,w))=\sqrt{s^2+t^2-2st\cos d_Z(z,w)}.
\end{equation*}
Denote the cone vertex by $o$. Let
\begin{equation*}
f:B_p(1/5)\longrightarrow B_o^{C(Z)}(1/5),\qquad
 f(x):=\begin{cases}
 o,&r(x)\le \sigma,\\
 (r(x),\bar x),&r(x)>\sigma.
\end{cases}
\end{equation*}
Here $\bar x=\gamma_{px}(\sigma)\in Z$ for  some(arbitrary chosen) minimizing geodesic $\gamma_{px}$ from $p$ to $x$. Then
\begin{equation}\label{eq:exactradial}
d(\bar x,x)=r(x)-\sigma,\qquad e(\bar x,x)=0.
\end{equation}
\textbf{Step 2. Prove $d_C(f(x),f(y))\le d(x,y)+C_N\sigma.$}

Take $x,y\in B_p(1/5)$. If $\min\{r(x),r(y)\}\le2\sigma$, then, 
\[
 d_C(f(x),f(y))\le r(x)+r(y)
 \le d(x,y)+2\min\{r(x),r(y)\}\le d(x,y)+4\sigma.
\]
Next we  assume $r(x),r(y)>2\sigma$.
Take a  minimizing geodesic $\gamma:[0,L]\to M$ joining $x$ to $y$. If there exists $u\in\gamma$ such that  $r(u)\le2\sigma$, then
\[
d_C(f(x),f(y))\le  r(x)+r(y)\le d(x,u)+d(u,y)+2r(u)=d(x,y)+4\sigma.
\]
So, we can assume that $r(\gamma(t))>2\sigma$ for any $t\in [0,L]$.
Choose
\[
 0=t_0<t_1<\cdots<t_N=L,\qquad x_i=\gamma(t_i),
\]
such that
$
  t_{i+1}-t_i\le\frac{\sigma^{3}}{50}$ and  $
 N\le C \sigma^{-3}.
$
For every $x_i$, define $z_i:=\bar x_i$. By \eqref{eq:exactradial}, $(z_i,x_i)$ and $(z_{i+1},x_{i+1})$ are two exact radial pairs. Thus \eqref{eq:four_K_estimate} holds, i.e., 
\[
|K(z_i,z_{i+1})-K(x_i,x_{i+1})|\le C_n\sigma^{18}.
\]
Write $R_i=r(x_i)$ and $d_i=d(x_i,x_{i+1})$. Then
\[
\frac{d(z_i,z_{i+1})^2}{2\sigma^2}
 =1-K(z_i,z_{i+1})\le 1-K(x_i,x_{i+1})+C_n\sigma^{18}=\frac{d_i^2-(R_i-R_{i+1})^2}{2R_iR_{i+1}}+C_n\sigma^{18}\le\frac{d_i^2}{8\sigma^2}+C_n\sigma^{18}.
\]
Hence
$
d(z_i,z_{i+1})\le\frac12d_i+C_n\sigma^{10}\le \sigma^{3}.
$
 Thus, $d_Z(z_i,z_{i+1})=\frac{d(z_i,z_{i+1})}{\sigma}$ and the local angle is exactly
$
\theta_i:=d_Z(z_i,z_{i+1})\le \sigma^2.
$
Note that 
\begin{equation*}
K(z_i,z_{i+1})=1-\frac{d(z_i,z_{i+1})^2}{2\sigma^2}=1-\frac{\theta_i^2}{2}\qquad \text{ and } \qquad    \left|\cos\theta_i-\left(1-\frac{\theta_i^2}{2}\right)\right|
 \le C\theta_i^4\le C\sigma^8.
\end{equation*}
By the definition of $K$ and $d_C$, we obtain
\begin{equation*}
\left|d(x_i,x_{i+1})^2-d_C(f(x_i),f(x_{i+1}))^2\right|=2R_iR_{i+1}\left|\cos\theta_i-K(z_i,z_{i+1})\right|\le C_n\sigma^8.
\end{equation*}
By  $|\sqrt A-\sqrt B|\le\sqrt{|A-B|}$, we get 
\[
\begin{aligned}
 \left|d(x_i,x_{i+1})-d_C(f(x_i),f(x_{i+1}))\right|
 \le C_n\sigma^4.
\end{aligned}
\]
Finally, using the triangle inequality in the cone and the minimizing property of $\gamma$,
\begin{equation*}
\begin{aligned}
 d_C(f(x),f(y))
 \le\sum_{i=0}^{N-1}d_C(f(x_i),f(x_{i+1}))\le\sum_{i=0}^{N-1}d(x_i,x_{i+1})+C_nN\sigma^4\le d(x,y)+C_n\sigma.
\end{aligned}
\end{equation*}

\textbf{Step 3. Prove $d(x,y)\le d_C(f(x),f(y))+C_n\sigma$.}

Take $x,y\in B_p(1/5)$. If $\min\{r(x),r(y)\}\le 2\sigma$, noting $d_C((s,z),(t,w))\ge |s-t|$, then  
\[d(x,y) \leq r(x)+r(y) \leq d_C(f(x),f(y)) + 2\min\{r(x),r(y)\}\leq d_C(f(x),f(y)) +4\sigma. \]
Thus we only need to  consider $R:=r(x)>2\sigma$ and $S:=r(y)>2\sigma$. Let $z=\bar x$ and $w=\bar y$.

If $d_Z(z,w)=\pi$, then 
$
 d_C(f(x),f(y))=\sqrt{R^2+S^2+2RS}=R+S\ge d(x,y).
$
Therefore, we can assume $0\le\Theta:=d_Z(z,w)<\pi$ without loss of generality.
By \eqref{eq:chain-metric}, for any $\eta>0$ with $\eta<\pi-\Theta$, one may choose  $ z=z_0,z_1,\ldots,z_N=w$ such that
\begin{equation}\label{eq:prech}
 \qquad d(z_{j-1},z_j)\le \sigma^{3},
 \qquad\frac1\sigma\sum_{j=1}^N d(z_{j-1},z_j)\le\Theta+\eta.
\end{equation}
If $\max\{d(z_{j-1},z_j),d(z_j,z_{j+1})\}$, delete $z_j$. Then, by induction, without loss of generality,  we can  assume 
$$\max{d(z_{j-1},z_j), d(z_j,z_{j+1})}\ge \frac{\sigma^3}{4}\qquad \text{ hence } \qquad N\le \frac{C}{\sigma^2}.$$
 Set
\begin{equation}\label{eq:theta_i}
 d_i:=d(z_{i-1},z_i),\qquad \theta_i:=\frac{d_i}{\sigma},
 \qquad \vartheta_j:=\sum_{i=1}^j\theta_i,\quad\vartheta_0=0.
\end{equation}
Then $0\le\theta_i\le \sigma^2$, $\Theta\le\vartheta_N\le\Theta+\eta<\pi$, $N\le C\sigma^{-2}$.
In $\mathbb R^2$, take the points in polar coordinates $P_0=(R,0)$, $P_N=(S,\vartheta_N)$. Join them by the straight chord $\overline{P_0P_N}$, and write
\begin{equation*}
 D_\eta:=|P_0-P_N|
 =\sqrt{R^2+S^2-2RS\cos\vartheta_N}.
\end{equation*}
As $\eta\to0$, we have $\vartheta_N\to\Theta$. So,
\begin{equation}\label{eq:Dto}
 D_\eta\longrightarrow \sqrt{R^2+S^2-2RS\cos\Theta}
 =d_C(f(x),f(y)).
\end{equation}
If there exists  $Q\in \overline{P_0P_N}$ such that $|Q|\le2\sigma$, then  $R\le|P_0-Q|+2\sigma$, $S\le|P_N-Q|+2\sigma$, and hence $ d(x,y)\le R+S\le D_\eta+4\sigma$. Letting $\eta\to 0$, we get $$d(x,y)\le d_C(f(x),f(y))+4\sigma.$$ 
So, in the following, we can  assume 
\begin{equation}\label{eq:chordaway} |Q|>2\sigma,\qquad \forall Q\in\overline{P_0P_N}.
\end{equation}
For each $j=0,\ldots,N$, the ray $\{(t,\vartheta_j):t>0\}\subset \mathbb{R}^2$ meets $\overline{P_0P_N}$ at exactly one point
\[
 P_j=(R_j,\vartheta_j),\qquad R_0=R,\quad R_N=S.
\]
Moreover, \eqref{eq:chordaway} implies 
$
 2\sigma<R_j\le\max\{R,S\}\le1/5.
$
Denote
\begin{equation*}
 D_i:=|P_{i-1}-P_i|=
 \sqrt{R_{i-1}^2+R_i^2-2R_{i-1}R_i\cos\theta_i}.
\end{equation*}
Then $\sum_{i=1} ^ND_i=D_\eta$. 
For $z_j\in Z$ and $1\le j\le N-1$, let  $T_j=R_j-\sigma\in[0,1/5]$,  there exists $\widehat x_j\in B_p(1/2)$ such that \eqref{eq:extension-radial-defect} holds, that is, 
\begin{equation}\label{eq:xhat}
|r(\widehat x_j)-R_j|+e(z_j,\widehat x_j)\le C_n\sigma^{43} \qquad \forall j\in [1,N-1].
\end{equation}
Denote $\hat{x}_0=x$, $\hat{x}_{N}=y$, then by $z_0=z=\bar{x}$ and $z_N=w=\bar{y}$, we know  
$$|r(\hat{x}_0)-R_0|+|r(\hat{x}_N)||+e(\hat{x}_0,z_0)+e(\hat{x}_N,z_N)=0.$$  Thus \eqref{eq:four_K_estimate}  holds for  $(z_{i-1},\widehat x_{i-1})$ and $(z_i,\widehat x_i)$. More precisely,  
\begin{equation}\label{eq:Kchord}
|K(\widehat x_{i-1},\widehat x_i)-K(z_{i-1},z_i)|\le C_n\sigma^{18}.
\end{equation}
By  $r(z_{i-1})=r(z_i)=\sigma$ and the definition  \eqref{eq:theta_i} of $\theta_i$, we get
\[
 K(z_{i-1},z_i)=1-\frac{d_i^2}{2\sigma^2}=1-\frac{\theta_i^2}{2}.
\]
By the definition of $K$ and $D_i$, we have 
\begin{align*}
 \bigl|d(\widehat x_{i-1},\widehat x_i)^2-D_i^2\bigr| &=\bigl|\,r(\widehat x_{i-1})^2+r(\widehat x_i)^2
 -2r(\widehat x_{i-1})r(\widehat x_i)K(\widehat x_{i-1},\widehat x_i)
 -\bigl(R_{i-1}^2+R_i^2-2R_{i-1}R_i\cos\theta_i\bigr)\,\bigr|\\
 &\le C_n\max_{0\le j\le N}|r(\hat{x}_{j})-R_j|+C_n\left(|K(\widehat x_{i-1},\widehat x_i)-K(z_{i-1},z_i)|+\left|\cos\theta_i-\left(1-\frac{\theta_i^2}{2}\right)\right|\right)\\
 &\le \sigma^{43}+C_n\sigma^{18}+C_n\sigma^8\le C_n\sigma^8,
\end{align*}
where we use \eqref{eq:xhat}\eqref{eq:Kchord} and $\theta_i\le \sigma^2$ in the last line. 
By  $|\sqrt A-\sqrt B|\le\sqrt{|A-B|}$, we get 
\begin{equation*}
|d(\widehat x_{i-1},\widehat x_i)-D_i|\le\sqrt{C_n\sigma^8}\le C_n\sigma^4.
\end{equation*}
 Thus
\begin{align}\label{eq:sumchord}
d(x,y)\le \sum_{i=1}^N d(\widehat x_{i-1},\widehat x_i)
 \le \sum_{i=1}^N D_i+C_nN\sigma^4\notag\le D_\eta+C_n \sigma^2.
\end{align}
Letting $\eta\to 0$ and using \eqref{eq:Dto} gives
\begin{equation}\label{eq:lower}
 d(x,y)\le d_C(f(x),f(y))+C_n\sigma^2\le d_C(f(x),f(y))+C_n\sigma.
\end{equation}
\textbf{Step 4.}
Now  Steps 2 and 3 implies that for all $x,y\in B_p(1/5)$,
\begin{equation}\label{eq:distortion}
\bigl|d(x,y)-d_C(f(x),f(y))\bigr|\le C_n\sigma.
\end{equation}
It remains to show that $f(B_p\left(\frac{1}{10}\right))$ is $C_n\sigma$-dense in $B_o^{C(Z)}\left(\frac{1}{10}\right)$.
Fix any $\xi=(R,z)\in B_o^{C(Z)}\left(\frac{1}{10}\right)$. If $R\le2\sigma$, take $x=p$, so $f(p)=o$, and hence
\[
d_C(\xi,f(p))=R\le2\sigma.
\]
So, we can assume $R>2\sigma$. Applying \eqref{eq:extension-radial-defect} for  $T=R-\sigma\in[0,1/5]$ gives $y\in B_p(1/2)$ such that
\begin{equation*}
|r(y)-R|+e(z,y)\le C_n\sigma^{43}.
\end{equation*}
Since $R\le\frac{1}{10}$,  $r(y)\le \frac{1}{10}+C_n\sigma^{43}<1/5$. Let $\bar y=\gamma_{py}(\sigma)$ be the  projection  fixed in the definition of  $f$. Since $e(z,y)\le C_n\sigma^{43}$, $ e(\bar y,y)=0$, applying \eqref{eq:four_K_estimate} to the two pairs $(z,y),(\bar y,y)$ gives $|K(z,\bar y)-K(y,y)|\le C_n\sigma^{18}$.
By $K(y,y)=1$ and $r(z)=r(\bar y)=\sigma$, the definition of $K$ gives
\begin{equation*}
d(z,\bar y)=\sqrt{2\sigma^2[1-K(z,\bar y)]}\le C_n\sigma^{10}.
\end{equation*}
Noting $C_n \sigma^{10}<\sigma^3$ for $\delta_n$ small enough,  the definition of $d_Z$ implies
\begin{equation*}
d_Z(z,\bar y)=\frac{d(z,\bar y)}\sigma\le C_n\sigma^9.
\end{equation*}
Therefore, the definition of $d_C$ gives
\begin{align*}
d_C\bigl((R,z),(R,\bar y)\bigr)=2R\sin\frac{d_Z(z,\bar y)}2\le C_n\sigma^9.
\end{align*}
Using $d_C(f(y),(R,\bar{y}))\le |r(y)-R|$, we know
\begin{align}
 d_C\bigl((R,z),f(y)\bigr)
 \le |r(y)-R|+d_C\bigl((R,z),(R,\bar y)\bigr)\le C_n\sigma^{43}+C_n\sigma^9\le C_n\sigma^9.\label{eq:surj-temp}
\end{align}
If $r(y)\le\frac{1}{10}$, take $x=y$ and we are done. If $r(y)>\frac{1}{10}$, choose $x=\gamma_{py}(\frac{1}{10})$. Then
\begin{equation*}
d(x,y)=r(y)-\frac{1}{10}\le |r(y)-R|\le C_n\sigma^{43}.
\end{equation*}
By  Step 2, we know 
$
d_C(f(x),f(y))\le d(x,y)+C_n\sigma\le C_n\sigma.
$
Combining with \eqref{eq:surj-temp} yields
\begin{equation*}
 d_C(\xi,f(x))\le d_C((R,z),f(y))+d_C(f(y),f(x))\le C_n\sigma.
\end{equation*}
 By the definition of pointed Gromov-Hausdorff distance, we know
$$d_{GH}\left(B_p\left(\frac{1}{10}\right),B^{C(Z)}_o\left(\frac{1}{10}\right)\right)
\leq C_n\sigma= C_n\delta^{\frac{1}{44(4n+1)}}.$$
\end{proof}

\section{Proof of Theorem \ref{main}}\label{sec:Proof of the Main Theorem}

Define
\begin{equation*}
 \Theta_R(p):=\inf_{C(Z)}\frac1R\,
 d_{GH}\left(B_p(R),B_o^{C(Z)}(R)\right).
\end{equation*}
Here the infimum is taken over all metric cone  $C(Z)$ whose cross section $(Z,d_Z)$ is a compact metric space with $\operatorname{diam}Z\le\pi$.

\begin{proposition}\label{prop:cone} Let $(M^{n\ge 3},g)$ be a non-parabolic manifold and with  $\operatorname{Ric}\geq0$, $p\in M$ and assume $G_p$ be the minimal normalized Green function satisfying \eqref{eq:intro_green_expansion}
 Then there exists $C_n<\infty$ such that
 \begin{equation}\label{eq:coneestimate}
 \Theta_R(p)\le C_n\bigl(H_pR^{n-2}\bigr)^{\alpha_n}, \qquad  \quad   \alpha_n:=[88(n+1)(4n+1)]^{-1}.
 \end{equation}
\end{proposition}
\begin{proof}
Let $\varepsilon_R:=H_pR^{n-2}$,  $\widetilde g:=(10R)^{-2}g$ and  $\widetilde r=(10R)^{-1}r$.
Then,
 \[
 \Theta_R(p)=\Theta_{\frac{1}{10}}^{\widetilde g}(p),
 \qquad \text{ where } \quad
 \Theta_s^{\widetilde g}(p)
 :=\inf_{C(Z)}\frac1s d_{GH}
 \left(B_p^{\widetilde g}(s),B_o^{C(Z)}(s)\right).
 \]
By scaling, the regular constant of the Green function of the normalized metric $\tilde{g}$ is
 $\widetilde H_p=(10R)^{n-2}H_p$.
 Set $\widetilde b=b/(10R)$, and take
 $F=\widetilde b^{\,2}/2$ and $F_0=\widetilde r^{\,2}/2$ with respect to $\widetilde g$. For the constant  $\delta_n$ in Theorem \ref{thm:three} and the constant $\varepsilon(n)$ in Theorem \ref{thm:quantitative_conical_function}, let $$\hat{\varepsilon}_n:=\min\left\{\varepsilon(n), \left(\frac{\delta_n}{C_n}\right)^{2(n+1)}\right\}.$$ If $\tilde{\varepsilon}_{12}:=\tilde{H}_p\cdot (12)^{n-2}\le \hat{\varepsilon}_n$, then by  Theorem \ref{thm:quantitative_conical_function} and Cauchy--Schwarz inequality, we have
 \begin{equation}\label{eq:e04}
 \delta:=\fint_{B_p(\frac{3}{2})}|\nabla F-\nabla F_0|_{\widetilde g}\dd V_{\widetilde g}+ \fint_{B_p(\frac{3}{2})}|\hess F-\widetilde g|_{\widetilde g}^2\dd V_{\widetilde g}
 \le C_n\widetilde\varepsilon_{12}^{1/[2(n+1)]}\le \delta_n.
 \end{equation}
Noting $F\in W^{2,2}_{\mathrm{loc}}(B_p(4))$. By smooth approximation, there exists $F_j\in C^\infty(B_p(4))$ such that $F_j\to F$ in $W^{2,2}_{\mathrm{loc}}$, while preserving \eqref{eq:e04} up to arbitrarily small error. Applying Theorem \ref{thm:three} to $F_j$ and then letting $j\to\infty$, we get 
 \[
 \begin{aligned}
 d_{GH}\left(B_p\left(\frac{1}{10}\right),B_o^{C(Z)}\left(\frac{1}{10}\right)\right)
 \le C_n\delta^{1/[44(4n+1)]}\le C_n\widetilde\varepsilon_{12}^{\,1/[88(n+1)(4n+1)]}=C_n\widetilde\varepsilon_{12}^{\alpha_n}.
 \end{aligned}
 \]
 If $\tilde{\varepsilon}_{12}\ge \hat{\varepsilon}_n$, we also have 
 \begin{align*}
  d_{GH}\left(B_p\left(\frac{1}{10}\right),B_o^{C(Z)}\left(\frac{1}{10}\right)\right)
 \le 2\le C_n\widetilde\varepsilon_{12}^{\alpha_n}.
 \end{align*}
Scaling back, we get  
\[\Theta_R(p) = \Theta_{\frac{1}{10}}^{\widetilde{g}}(p)\leq C_n \widetilde{\varepsilon}_{12}^{\alpha_n}\leq C_n (H_pR^{n-2})^{\alpha_n}.\]
\end{proof}

\begin{proof}[Proof of Theorem \ref{main}]
 Let
 $r_k=Re^{-k}$.
 Then
$
 H_pr_k^{n-2}=\varepsilon_Re^{-k(n-2)}.
 $
 By Proposition~\ref{prop:cone},
 \[
 \Theta_{r_k}(p)\le C_n\varepsilon_R^{\alpha_n}
 e^{-k(n-2)\alpha_n}.
 \]
 By the definition of pointed Gromov-Hausdorff distance, for any $r>0$, there holds $$d_{GH}\left(r^{-1}B_p(r),(er)^{-1}B_p(er)\right)
 \le(e+3)\Theta_{er}(p).$$
 So,
 \[
 d_{GH}\left(r_{k+1}^{-1}B_p(r_{k+1}),
 r_k^{-1}B_p(r_k)\right)
 \le C_n\varepsilon_R^{\alpha_n}e^{-k(n-2)\alpha_n}.
 \]
 Summing over $k$, we get
 \[
 d_{GH}\left(R^{-1}B_p(R),r_N^{-1}B_p(r_N)\right)
 \le C_n\varepsilon_R^{\alpha_n}
 \sum_{k=0}^{N-1}e^{-k(n-2)\alpha_n}
 \le C_n\varepsilon_R^{\alpha_n}.
 \]
 Since $M$ is smooth, $
r_N^{-1}B_p(r_N)\longrightarrow B_0^{\R^n}(1)$. Letting $N\to\infty$, we obtain 
\begin{align*}
    \frac1R d_{\mathrm{GH}} \bigl(B_p(R),B^{\mathbb{R}^n}_0(R)\bigr) \le C_n\bigl(H_pR^{n-2}\bigr)^{\alpha_n}.
\end{align*}.
\end{proof}

\appendix
\section{}
\label{app:auxiliary}

\begin{proposition}\label{prop:ricci-radial-jet}
Let $(M^n,g)$ be a smooth Riemannian manifold with $n\geq 3$,
let $p\in M$ and  $m\in \mathbb{N}$.
Assume that $\Ric\geq 0$
in a neighborhood of $p$.
If, 
\begin{equation}
    (\nabla^k\Ric)_p
    (\underbrace{\theta,\ldots,\theta}_{k\text{ derivative slots}};
    \theta,\theta)
    =
    0,
    \qquad \forall \theta\in T_pM, \quad  
    0\leq k\leq m,
    \label{eq:radial-Ricci-jet-vanishing}
\end{equation}
then 
    $ (\nabla^k\Ric)_p=0$ for $0\leq k\leq m.$
\end{proposition}
\begin{proof}
We argue by induction on $k$.
For $k=0$,  $\Ric_p(\theta,\theta)=0$
for every $\theta\in T_pM$ means
$\Ric_p=0$.
Now let $1\leq k\leq m$, and assume that
\[
    (\nabla^j\Ric)_p=0,
    \qquad
    0\leq j<k.
\]
Choose orthonormal normal coordinates $x=(x^1,\ldots,x^n)$ centered at $p$, so that $p$ corresponds to the origin, $\Gamma_{ij}^k(p)=0$ and write
\[
    R_{ij}(x)
    :=
    \Ric_x(\partial_i,\partial_j).
\]
The difference between covariant derivatives and ordinary partial
derivatives consists only of products involving lower-order
derivatives of the Ricci tensor and connection coefficients.
Therefore, by the induction hypothesis, all partial derivatives of
$R_{ij}$ of order less than $k$ vanish at the origin, and the
$k$-th covariant derivatives agree with the $k$-th ordinary partial
derivatives there.
Let $A_{ij}$ denote the homogeneous degree-$k$ Taylor term of $R_{ij}$:
\[
    A_{ij}(x)
    =
    \frac{1}{k!}
    \sum_{a_1,\ldots,a_k=1}^n
    (\partial_{a_1}\cdots\partial_{a_k}R_{ij})(0)
    x^{a_1}\cdots x^{a_k}.
\]
By smoothness,
\begin{equation}
    R_{ij}(x)
    =
    A_{ij}(x)+O(|x|^{k+1}),
    \qquad
    \partial_\ell R_{ij}(x)
    =
    \partial_\ell A_{ij}(x)+O(|x|^k).
    \label{eq:Ricci-Taylor}
\end{equation}
Let $A(x):=(A_{ij}(x))$, then by $\mathrm{Ric}\ge 0$ and the homogeneity of $A$, we know   $$A(x)=\lim_{t\to 0}t^{-k}A(\frac{x}{t})=\lim_{t\to 0}t^{-k}(R_{ij}(tx))+O(t)\ge 0.$$ For a unit vector $\theta\in T_pM$, let $\gamma_\theta(t)=\exp_p(t\theta)$. Then
\[
    \left.
    \frac{d^k}{dt^k}
    \right|_{t=0}
    \Ric_{\gamma_\theta(t)}
    \bigl(\dot\gamma_\theta,\dot\gamma_\theta\bigr)
    =
    (\nabla^k\Ric)_p
    (\underbrace{\theta,\ldots,\theta}_{k\text{ derivative slots}};
    \theta,\theta).
\]
On the other hand, by \eqref{eq:Ricci-Taylor}, this derivative equals
\[
    k!
    \sum_{i,j}
    A_{ij}(\theta)\theta^i\theta^j.
\]
Thus, by \eqref{eq:radial-Ricci-jet-vanishing} and definitio of $A_{ij}(x)$, we get 
\begin{equation}
    x^TA(x)x=0
    \qquad
    \text{for every }x\in\mathbb R^n.
    \label{eq:xAx-zero}
\end{equation}
Since $A(x)$ is positive semidefinite,
\eqref{eq:xAx-zero} implies $A(x)x=0$.
Equivalently,
\begin{equation}
    \sum_j A_{ij}(x)x^j=0
    \qquad
    \text{for every }i.
    \label{eq:Ax-zero}
\end{equation}
Set $T(x)
    :=
    \operatorname{tr}A(x)
    =
    \sum_i A_{ii}(x)$.
The contracted Bianchi identity gives
\begin{equation}
    \nabla^i R_{ij}
    =
    \frac12\,\partial_j\operatorname{Scal}.
    \label{eq:contracted-Bianchi}
\end{equation}
In normal coordinates,
\[
    g^{i\ell}
    =
    \delta^{i\ell}+O(|x|^2),
    \qquad
    \partial_jg^{i\ell}
    =
    O(|x|),
    \qquad
    \Gamma^a_{ij}
    =
    O(|x|).
\]
Together with \eqref{eq:Ricci-Taylor}, this gives
\[
    \nabla^iR_{ij}
    =
    \sum_i\partial_iA_{ij}
    +
    O(|x|^k), \qquad 
    \partial_j\operatorname{Scal}
    =
    \partial_jT
    +
    O(|x|^k).
\]
The second identity follows directly from $\operatorname{Scal} = g^{ab}R_{ab}$ and the product rule.
Substituting these expressions into
\eqref{eq:contracted-Bianchi}, we find that
\[
    \sum_i\partial_iA_{ij}
    -
    \frac12\,\partial_jT
\]
is a homogeneous polynomial of degree $k-1$ which is also
$O(|x|^k)$. Hence it vanishes identically. Therefore
\begin{equation}
    \sum_i\partial_iA_{ij}
    =
    \frac12\,\partial_jT.
    \label{eq:div-A}
\end{equation}
Differentiate \eqref{eq:Ax-zero} with respect to $x^i$ and sum over
$i$. Using \eqref{eq:div-A}, we obtain
\begin{align*}
    0
    =
    \sum_{i,j}
    \partial_i\bigl(A_{ij}(x)x^j\bigr)
    =
    \sum_j
    x^j
    \sum_i\partial_iA_{ij}
    +
    \sum_iA_{ii}
    =
    \frac12
    \sum_jx^j\partial_jT
    + T.
\end{align*}
Since $T$ is a homogeneous polynomial of degree $k$, by the Euler's identity
\[
    \sum_jx^j\partial_jT=kT,
\]
we know  $0 =\left(1+\frac{k}{2}\right)T$.
Thus $T\equiv0$.
Since $A(x)\ge 0$, it follows that $A(x)\equiv0$.
Hence all $k$-th order partial derivatives of $R_{ij}$ vanish at the
origin. This completes the induction.
\end{proof}

\begin{lemma}
\label{lem:jacobi-jet-isomorphism}
Let $(M,g)$ be a smooth Riemannian manifold and let $p\in M$.
For $j\geq0$, define the operator at $p$ by
\[
    \mathcal{R}_p^{(j)}(v)w
    :=
    \bigl(\nabla^j_{v,\ldots,v}\mathrm{Rm}\bigr)_p(w,v)v,
    \qquad
    v,w\in T_pM.
\]
Suppose that for some $N\geq0$,
\begin{equation}\label{eq:jet_lemma}
    \mathcal{R}_p^{(j)}(v)=0
    \qquad
    \text{for all }v\in T_pM,\quad 0\leq j\leq N.
\end{equation}  
Then $
    \nabla^j\operatorname{Rm}(p)=0
   $ for any $0\leq j\leq N.$
\end{lemma}
\begin{proof} By \eqref{eq:jet_lemma} and \eqref{eq:Taylor-R-theta}, 
\begin{equation}\label{eq:jet_R_small}
    R_{\theta} = O(t^{N+1}) \quad \text{holds for } \theta \in \mathbb{S}^{n-1} \subset T_pM \text{ uniformly} .
\end{equation}
For every $w\in T_pM$, let $Y_w$ be the Jacobi field 
with $Y_w(0)=0$, $\frac{DY_w}{dt}(0)=w$.
Define $J_\theta(t)w :=P_t^{-1}Y_w(t)$. Then
\[
    J_\theta''(t)
    +
    R_\theta(t)J_\theta(t)
    =
    0, \qquad  J_\theta(0)=0,
    \qquad
    J_\theta'(0)=I.
\]
Integration by parts gives
\begin{equation*}
    J_\theta(t)
    =
    tI
    -
    \int_0^t
    (t-s) R_\theta(s)J_\theta(s)\,ds.
\end{equation*}
By  $J_\theta(s)=O(s)$.
and 
\eqref{eq:jet_R_small}, we obtain
\begin{equation}
    J_\theta(t)
    =
    tI+O(t^{N+4}),
    \qquad
    \frac{J_\theta(t)}{t}
    =
    I+O(t^{N+3}).
    \label{eq:jacobi-euclidean-approximation}
\end{equation}
Choose an orthonormal basis $(e_i)$ at $p$ and use
$x\to \exp_p(x)$ as normal coordinates. Along $x=t\theta$, the differential of the exponential map satisfies
\begin{equation*}
    (d\exp_p)_{t\theta}(w)
    =
    P_t
    \frac{J_\theta(t)w}{t}.
\end{equation*}
Therefore the coordinate coefficients of the metric are
\begin{align*}
    g_{ij}(t\theta)
    =
    \left\langle
        (d\exp_p)_{t\theta}e_i,
        (d\exp_p)_{t\theta}e_j
    \right\rangle
    =
    \left\langle
        \frac{J_\theta(t)}{t}e_i,
        \frac{J_\theta(t)}{t}e_j
    \right\rangle_p
    =
    \delta_{ij}
    +
    O(t^{N+3}).
\end{align*}
That is, 
$
    g_{ij}(x)
    =
    \delta_{ij}
    +
    O(|x|^{N+3}).
$
 Consequently,
\[
    \nabla^j\operatorname{Rm}(p)
    =
    0,
    \qquad
    0\leq j\leq N.
\]
\end{proof}

\bibliographystyle{amsplain}
\bibliography{references}
\end{document}